\documentclass[12pt]{article}

\usepackage{calrsfs}

\usepackage{setspace}
\usepackage[all]{xy}
\usepackage{amssymb,amsmath,amsthm}
\usepackage{stmaryrd}

\newcommand{\lmod}{\mbox{-}\mathrm{mod}\!}

\newcommand{\Hom}{\mathrm{Hom}}
\newcommand{\GL}{\mathrm{GL}}
\newcommand{\K}{\mathbb{K}}

\newcommand{\Z}{\mathbb{Z}}
\newcommand{\N}{\mathbb{N}}
\newcommand{\Q}{\mathbb{Q}}
\newcommand{\F}{\mathbb{F}}

\newcommand{\im}{\operatorname{Im}}

\newcommand{\X}{\mathbb{X}}
\newcommand{\diag}{\operatorname{diag}}
\newcommand{\wt}{\operatorname{wt}}

\newcommand{\End}{\mathrm{End}}
\newcommand{\Tor}{\mathrm{Tor}}

\newcommand{\id}{\mathrm{id}}

\newcommand{\involution}{\mathcal{J}}

\newcommand{\dominate}{\ensuremath{\trianglerighteq}}
\newcommand{\strdominate}{\ensuremath{\triangleright}}
\newcommand{\tableau}{\mathcal{T}}
\newcommand{\rs}{\mathfrak{r}}
\newcommand{\bfr}{\mathbf{r}}
\newcommand{\bfn}{\mathbf{n}}
\renewcommand{\sf}{\mathfrak{S}}

\newtheorem{theorem}{Theorem}[section]
\newtheorem{proposition}[theorem]{Proposition}
\newtheorem{corollary}[theorem]{Corollary}

\newtheorem{remark}[theorem]{Remark}
\theoremstyle{definition}
\newtheorem{definition}[theorem]{Definition}
\theoremstyle{remark}
\author{Ana Paula Santana and Ivan Yudin \thanks{The second author's work is supported by the FCT Grant
SFRH/BPD/31788/2006. Financial support by CMUC/FCT gratefully acknowledged by both
authors.}\\
Department of Mathematics \\
University of Coimbra\\
Coimbra, Portugal
}

\title{Characteristic-free resolutions of Weyl and Specht modules}

\begin{document}
\maketitle
\begin{abstract}
	We construct  explicit resolutions of Weyl modules by divided powers and
	of co-Specht modules by permutational modules. 
	We also prove a conjecture by Boltje-Hartmann~\cite{boltje} on
	resolutions of co-Specht modules.
\end{abstract}
\section*{Introduction}

Schur algebras are fundamental tools in the representation theory of the general
linear group $\GL_n\left( R \right)$ and of the symmetric group. In fact,
over infinite fields, the category of homogeneous polynomial representations of
degree $r$ of $\GL_n\left( R \right)$ is equivalent to the category of finite-dimensional modules for the Schur algebra $S_R\left( n,r \right)$. If $r\le n$,
the use of the Schur functor (see~\cite[\S 6]{green}) allows us to relate these
categories to the category of finite-dimensional representations of the
symmetric group $\Sigma_r$. 

Introduced by I. Schur in his doctoral dissertation~\cite{schur1} in 1901, for
the field of complex numbers, Schur algebras were generalized for arbitrary
infinite fields by J.A.~Green in~\cite{green}. The Schur algebra $S_R\left( n,r
\right)$, for a commutative ring $R$ with identity, was introduced by K.~Akin
and D.~Buchsbaum in~\cite{ab2} and by Green in~\cite{greencombinatorics}.

If $R$ is a noetherian commutative ring, $S_R\left( n,r \right)$ is
quasi-hereditary. So it is natural to ask for the construction of projective
resolutions of Weyl modules, which are the standard modules in that case. 

In their work on characteristic-free representation theory of the general linear
group~\cite{abw2},~\cite{ab1},~\cite{ab2}, Akin, Buchsbaum, and Weyman study the
problem of constructing resolutions of Weyl modules in terms of direct sums of
tensor products of divided powers of $R^n$. Moreover, they ask for these
resolutions to be finite and universal (defined over the integers). 

This task was accomplished for Weyl modules associated with two-rowed partitions
in~\cite{ab1} and three-rowed (almost) skew-partitions in~\cite{rota}. Using
induction, and assuming these resolutions are known for m-rowed (almost)
skew-partitions, for all $m<n$, such resolutions are presented
in~\cite{rota} for all $n$-rowed (almost) skew-partitions. But, in general, no
explicit description of such complexes is known.  

In this paper we use the theory of Schur algebras to give an answer to the above
construction problem for an arbitrary partition. 

Denote by $\Lambda\left( n;r \right)$ (respectively $\Lambda^+\left( n;r
\right)$) the set of all compositions (respectively partitions) of $r$ into at
most $n$ parts. For each $\lambda\in \Lambda^+\left( n;r \right)$ write
$W^R_\lambda$ for the Weyl module over $S_R\left( n,r \right)$ associated to
$\lambda$. 
The Schur algebra $S_R\left( n,r \right)$ has a decomposition of the identity as
a sum of orthogonal idempotents
$\xi_\mu$, where $\mu\in \Lambda\left( n;r \right)$. The projective module
$S_R\left( n,r \right)\xi_\mu$ is isomorphic as a
$\GL_n\left( R \right)$-module to the tensor product of divided powers
$D_{\mu_1}\left( R^n \right)\otimes_R \dots \otimes_R D_{\mu_n}\left( R^n
\right)$. Hence, the construction of a universal projective resolution of the
Weyl module $W^R_\lambda$ in terms of direct sums of the projective modules
$S_R\left( n,r \right)\xi_{\mu}$ will give an answer to the problem posed by Akin
and Buchsbaum. 

The Borel-Schur algebra $S^+_R\left( n,r \right)$ is a subalgebra of $S_R\left(
n,r \right)$ introduced by Green in~\cite{greencombinatorics} (see
also~\cite{green2}). It is an algebra with interesting properties. All the
idempotents $\xi_\mu$ are elements of $S_R^+\left( n,r \right)$ and
we have
\begin{align*}
S_R\left(
n,r \right)\otimes_{S_R^+\left( n,r \right)} S^+_R\left( n,r \right)\xi_\mu\cong
S_R\left( n,r
\right)\xi_\mu.
\end{align*}
	Moreover, for every $\lambda\in \Lambda^+\left( n,r
\right)$ there exists a rank one $S^+_R\left( n,r \right)$-module $R_\lambda$
such that $S_R\left( n,r \right)\otimes_{S^+_R\left( n,r \right)} R_\lambda\cong
W^R_\lambda$.

Woodcock~\cite{woodcock} proved that if $R$ is an infinite field then the
modules $R_\lambda$, $\lambda\in \Lambda^+\left( n,r \right)$, are acyclic with
respect to the induction functor $S_R\left( n,r \right)\otimes_{S^+_R\left( n,r
\right)} - $. The first author proved in~\cite{aps} that the modules $S^+_R\left( n,r
\right)\xi_\mu$ are principal projective modules in this case.
Therefore, applying the induction functor to a projective resolution of $R_\lambda$ we
obtain a resolution 
of
$W^R_\lambda$ by direct sums of modules $S_R\left( n,r \right)\xi_\mu$.

In Theorem~\ref{woodcock} we show that the modules $R_\lambda$ are 
acyclic with respect to the induction functor $S_R\left( n,r
\right)\otimes_{S^+_R\left( n,r \right)} - $ in the case of an arbitrary commutative
ring $R$. Then we construct a universal resolution of $R_\lambda$
 by direct sums of modules $S^+_R\left(
n,r \right)\xi_\mu$, $\mu\in \Lambda\left( n;r \right)$. 
Applying the induction functor we obtain a universal resolution of the Weyl module
by direct sums of the modules $S_R\left( n,r \right)\xi_\mu$, $\mu\in \Lambda\left(
n;r \right)$.

For $\lambda\in \Lambda^+\left( n;r \right)$, a complex $\left(
\widetilde{C}^{\lambda}_k, k\ge -1 \right)$
was constructed in~\cite{boltje}. In this complex,  $\widetilde{C}^{\lambda}_{-1}$ is the co-Specht
module associated with $\lambda$ and  $\widetilde{C}^{\lambda}_k$ are
permutational modules over $\Sigma_r$ for $k\ge 0$. Boltje and Hartmann
conjectured that $\widetilde{C}^{\lambda}_*$ is exact and thus gives a
permutational resolution for the co-Specht module. 

In Theorem~\ref{thm:boltje} we will show that if we apply the Schur functor
 to our resolution of the Weyl module $S_R\left( n,r
\right)\otimes_{S^+_R\left( n,r \right)}R_\lambda$, we obtain
$\widetilde{C}^{\lambda}_*$. As the Schur functor is exact, this proves
Conjecture~3.4 in~\cite{boltje}.


It should be noted that several other results are known if we look for
resolutions of Weyl modules by divided powers in the case $R$ is a field. 
In \cite{z}, \cite{akin} and \cite{woodcock2} such resolutions are described if
$R$ is a field of characteristic zero. Their proofs use the BGG-resolution. Also in
this case, resolutions are given in~\cite{artale}, for three-rowed (almost)
skew-partitions using the technique developed in~\cite{rota}. 

If $R$ is a field of positive characteristic, projective resolutions of the
simple $S^+_R\left( n,r \right)$-modules $R_\lambda$ are constructed in~\cite{yudin1}, for the cases $n=2$ and $n=3$. Using the
induction functor, one gets resolutions of the corresponding Weyl modules by
divided powers for $n=2$, $3$. 

The present paper is organized as follows. 
In Section~\ref{combinatorics} we introduce the combinatorial notation. 
In Section~\ref{schur} we recall the definition of Schur algebra and give  
a new version of the formula for the product of two basis elements of $S_R\left(
n,r \right)$, which seems to be well suited to our work with divided powers. 

Section~\ref{bar_constr} is technical and is included for the convenience of the
reader. It explains what the reduced bar construction looks like for augmented
algebras in the monoidal category of $S$-bimodules, where $S$ is an arbitrary
commutative ring. 

Section~\ref{borel} is dedicated to the Borel-Schur subalgebra
$S^+_R\left( n,r \right)$. We  apply the reduced bar construction to obtain a
projective resolution for every rank one module $R_\mu$, $\mu\in\Lambda\left(
n;r \right)$. 

Section~\ref{induction} contains the main result (Theorem~\ref{woodcock}) of the
paper.

In Sections~\ref{symmetric} and~\ref{boltjehartmann} we explain how our
results prove Conjecture~3.4 in~\cite{boltje}. 

In an appendix we show that the Schur algebra definition used in this paper is
equivalent to the one given in~\cite{greencombinatorics}. Then we construct an
explicit isomorphism between $D_\lambda\left( R^n \right)$ and $S_R\left(
n,r \right)\xi_\lambda$, as this seems to be unavailable in published form. We also
recall the theory of divided powers. 

\section{Combinatorics}
\label{combinatorics}
In this section we collect the  combinatorial notation used in the paper. We will 
give general definitions which include as partial cases the usual tools in the
subject such as multi-indices, compositions, etc. 

We denote by $R$ a commutative ring with identity $e$,  and $n$ and $r$ are
arbitrary fixed positive integers. For any natural number $s$ we denote by
$\mathbf{s}$ the set $\left\{ 1,\dots,s \right\}$ and by $\Sigma_s$ the
symmetric group on $\mathbf{s}$. Given a finite set $X$, we write:
\begin{itemize}
	\item $\mu= \left( \mu_x \right)_{x\in X}$ and $|\mu| = \sum_{x\in
		X}\mu_x$, for each map $\mu\colon X\to \N_0$ given by $x\mapsto
		\mu_x$. 
	\item $\Lambda\left( X,r \right) := \left\{\, \mu\colon X\to \N_0
		\,\middle|\, \left|\mu\right| = r
		\right\}$.
	\item $\wt\left( u \right)\in \Lambda\left( X,r \right)$, for the map
		defined by 
		$$
		\wt\left( u \right)_x = \# \left\{\, s \,\middle|\,  u_s = x, \
		s= 1,\dots, r
		\right\},
		$$
		for each $u\in X^r$.
\end{itemize}
The symmetric group $\Sigma_r$ acts on the right of $X^r$ in the usual way: 
$$
\left( x_1,\dots,
x_r
\right)\sigma = \left( x_{\sigma\left( 1 \right)}, \dots, x_{\sigma\left( r
\right)} \right).
$$
Identifying $\wt\left( u \right)$ with the $\Sigma_r$-orbit
of $u\in X^r$, we can think of $\Lambda\left( X,r \right)$ as the set of
$\Sigma_r$-orbits on $X^r$. We will write $u\in \omega$ if $\wt\left( u
\right) = \omega$. 

	Next we consider several particular cases of the definitions given above. We
write $I\left( n,r \right)$ for $\bfn^r$. The elements of $I\left( n,r
\right)$ are called \emph{multi-indices} and will be usually denoted by the
letters $i$, $j$. We identify the sets $\left( \bfn\times \bfn
\right)^r$ and $I\left( n,r \right)\times I\left( n,r \right)$ via the map
$$
\left( \left( i_1,j_1 \right), \dots, \left( i_r,j_r \right) \right)\mapsto
\left( \left( i_1,\dots, i_r \right), \left( j_1,\dots, j_r \right) \right).
$$
Similarly $\left( \bfn\times \bfn\times \bfn  \right)^r$ will be identified with
$I\left( n,r \right)\times I\left( n,r \right)\times I\left( n,r \right)$. 

The sets $\Lambda\left( \bfn,r \right)$, $\Lambda\left( \bfn\times\bfn;r \right)$,
and $\Lambda\left( \bfn\times\bfn \times \bfn;r \right)$ will be denoted by
$\Lambda\left( n;r \right)$, $\Lambda\left( n,n;r \right)$, and $\Lambda\left(
n,n,n;r \right)$, respectively. We can think of the elements of
$\Lambda\left( n;r \right)$ as \emph{the compositions} of $r$ into at most $n$ parts,
and we will write $\Lambda^+\left( n;r \right)$ for those $\left(
\lambda_1,\dots, \lambda_n
\right)\in \Lambda\left( n;r \right)$ that verify $\lambda_1\ge \dots\ge
\lambda_n$ (the \emph{partitions} 
of $r$ into at most $n$ parts). The elements of $\Lambda\left( n,n;r \right)$
are functions from $\bfn\times \bfn$ to $\N_0$ and can be considered as
$n\times n$ matrices of non-negative integers $\left( \omega_{s,t}
\right)_{s,t=1}^n$ such that $\sum_{s,t=1}^n \omega_{st} = r$. Similarly, the
elements of $\Lambda\left( n,n,n;r \right)$ can be interpreted in terms of
$3$-dimensional $n\times n\times n$ tensors. 

Note that the weight function on $I\left( n,r \right)$ coincides with the one defined
in~\cite{green}. Now, given $\omega= \left( \omega_{st} \right)_{s,t=1}^n \in
\Lambda\left( n,n;r \right)$ and $\theta= \left( \theta_{s,t,q}
\right)_{s,t,q=1}^n \in \Lambda\left(
n,n,n;r \right)$ we will define $\omega^1$, $\omega^2\in \Lambda\left(
n;r \right)$ and $\theta^1$, $\theta^2$, $\theta^3\in \Lambda\left( n,n;r
\right)$ by 
\begin{align*}
	\left( \omega^1 \right)_t, &= \sum_{s=1}^n \omega_{st}, & \left( \omega^2
	\right)_s & = \sum_{t=1}^n \omega_{st}; \\ \left( \theta^1
	\right)_{tq} & = \sum_{s=1}^n \theta_{stq}, & \left( \theta^2
	\right)_{sq} & = \sum_{t=1}^n \theta_{stq}, & \left( \theta^3
	\right)_{st} & = \sum_{q=1}^n \theta_{stq}. 
\end{align*}
\begin{remark}
	\label{commutative}
	If $i$, $j$, $k\in I\left( n,r \right)$ then it can be seen from the
	definition of the weight function that
	\begin{align*}
		\wt\left( i,j \right)^1 &= \wt\left( j \right), & \wt\left( i,j
		\right)^2, & = \wt\left( i \right);\\
		\wt\left( i,j,k \right)^1 & = \wt\left( j,k \right), & \wt\left(
		i,j,k \right)^2 & = \wt\left( i,k \right), & \wt\left( i,j,k
		\right)^3 &= \wt\left( i,j \right). 
	\end{align*}
\end{remark}

\section{The Schur algebra $\mathbf{S_R\left( n,r \right)}$}
\label{schur}
In this section we recall the definition of the Schur algebra  $S_R\left( n,r
\right)$ over a commutative
ring and give a new version of the formula for the product of two basis elements of
this algebra (cf. \cite[(2.3.b)]{green} and \cite[(2.6)]{green2}). 

Let  $\left\{\, e_s \,\middle|\,1\le s\le n 
\right\}$ be the standard  basis of $R^n$. For every $i\in I\left( n,r \right)$
define $e_i = e_{i_1}\otimes \dots \otimes e_{i_r}$. Then 
$\left\{\, e_i\,\middle|\, i\in I(n,r) \right\}$ is a basis for $\left( R^n
\right)^{\otimes r}$  and $\End_R\left( \left( R^n \right)^{\otimes r} \right)$  
has basis
 $\left\{ \, e_{i,j}\,\middle|\, i,j\in I(n,r) \right\}$, where the map
$e_{i,j}$ is defined by 
$$
e_{i,j}e_k := \delta_{jk} e_i, \ i,j,k\in I(n,r).
$$
The action of $\Sigma_r$ on $I\left( n,r \right)$ extends on $\left( R^n
\right)^{\otimes r}$ to
  $e_i \sigma :=
  e_{i\sigma}$ and turns $\left( R^n \right)^{\otimes r}$ into a right $R \Sigma_r $-module.

\begin{definition}
	The Schur algebra $S_R\left( n,r \right)$ is the endomorphism algebra of
	$\left( R^n \right)^{\otimes r}$ in the category of $R \Sigma_r $-modules. 
\end{definition}
The action of $\Sigma_r$ on $\left( R^n \right)^{\otimes r}$ induces a
$\Sigma_r$-action
on $\End_R\left( \left( R^n \right)^{\otimes r} \right)$ by $\left( f\sigma \right)\left( v \right) := f\left(
v\sigma^{-1} \right)\sigma$. On the basis elements this action is 
$e_{i,j}\sigma = e_{i\sigma , j\sigma}$. 
We have, by  Lemma 2.4 in~\cite{feit}, $S_R\left( n,r \right) \cong \left(\End_R\left(
R^n \right)^{\otimes r}\right)^{\Sigma_r}$. 

For $\omega \in \Lambda\left( n,n;r \right)$ define $\xi_\omega\in S_R\left(
n,r \right)$ by
$$
\xi_{\omega} := \sum_{\left( i,j \right)\in \omega }e_{i,j}. 
$$
Since $\Lambda\left( n,n;r \right)$ is identified with the set of
$\Sigma_r$-orbits on $I\left( n,r \right)\times
I\left( n,r \right)$ via $\wt$,  the set  $\left\{\, \xi_{\omega} \,\middle|\, \omega\in \Lambda\left( n,n;r
\right) \right\}$ is a basis for $S_R\left( n,r \right)$. 

\begin{remark}
	The definition of Schur algebra we use is equivalent to the one given by
	J.A.~Green in~\cite{greencombinatorics} (cf. Theorem~\ref{equivalence}). 
	Note that Green writes $\xi_{i,j}$ where we have $\xi_{\wt\left( i,j
	\right)}$ for $i$, $j\in I\left( n,r \right)$. In the case $R$ is an
	infinite field, this definition of $S_R\left( n,r \right)$ is also
	equivalent to the one given in~\cite{green}.  
\end{remark}
For $\lambda\in \Lambda\left( n;r \right)$ we write $\xi_{\lambda}$ for
$\xi_{\diag\left( \lambda \right)}$. It is immediate from the definition, that
$\sum_{\lambda\in \Lambda\left( n;r \right)} \xi_{\lambda}$ is an orthogonal
idempotent decomposition of the identity of $S_R\left( n,r \right)$.

Next we
will deduce a product formula for any two basis elements of $S_R\left( n,r
\right)$. 
Let $\theta \in \Lambda\left( n,n,n;r \right)$. We write  $\left[ \theta
\right]\in \N_0$ for the product of binomial coefficients: 
$$
\prod_{s,t=1}^n \binom{\left( \theta^2 \right)_{st}}{\theta_{s1t},
\theta_{s2t}, \dots, \theta_{snt}}. 
$$
 
\begin{proposition}
	\label{mult}
	For any $\omega$, $\pi\in \Lambda\left( n,n;r \right)$ we have
	\begin{equation}	
		\label{eq:mult}
		\xi_{\omega}\xi_{\pi} = \sum_{\theta\in \Lambda\left( n,n,n;r
		\right):\
		\theta^3= \omega,\ \ \theta^1=\pi}
		\left[ \theta \right]\xi_{\theta^2}. 
	\end{equation}
\end{proposition}
\begin{proof}
	We have 
	\begin{align*}
		\xi_{\omega}\xi_{\pi} = \left( \sum_{\left( i,j \right)\in
		\omega}e_{i,j} \right) \left( \sum_{\left( l,k \right)\in
		\pi}e_{l,k} \right) = \sum_{\left( i,j,k \right):\ \left( i,j
		\right)\in \omega,\ \left( j,k \right)\in \pi} e_{i,k}. 
	\end{align*}
	Note that the right hand side of the above formula is
	$\Sigma_r$-invariant.

	From Remark~\ref{commutative} it follows
	that  if $\wt\left( i,j,k \right) = \theta\in \Lambda\left( n,n,n;r \right)$, then $\left(
	i,j \right)\in \omega$ if and only if $\omega=\theta^3$,
	and $\left( j,k \right)\in \pi$ if and only if  $\pi=
	\theta^1$. Thus
	\begin{align*}
		\xi_{\omega}\xi_{\pi} = \sum_{\theta\colon \theta^3=\omega,\
		\theta^1=\pi}\ \ \ \sum_{\left( i,j,k \right)\in \theta}
		e_{i,k}.
	\end{align*}
	Let us fix $\theta\in \Lambda\left( n,n,n;r \right)$ such that
	$\theta^1= \pi$ and $\theta^3 = \omega$. 
	Then $\left( i,j,k \right)\in \theta$ implies that $\left( i,k
	\right)\in \theta^2$. Thus $\sum_{\left( i,j,k \right)\in \theta }
	e_{i,k}$ is a multiple of $\xi_{\theta^2}$, where the multiplicity is
	given by the number
	$$
	\#\left\{\, j \,\middle|\, \left( i,j,k \right)\in \theta \right\}
	$$
	for each pair $\left( i,k \right)\in \theta^2$. 
	Now we fix $\left( i,k \right)\in\theta^2$ and define the sets
	$X_{sq}$ by 
	$$
	X_{sq} := \left\{\, 1\le t \le r \,\middle|\, i_t =s,\ k_t=q  \right\}.
	$$
	For every $j$ such that $\left( i,j,k \right)\in \theta$ and $1\le
	s,q\le n$ we define the function
	\begin{align*}
		J_{sq}\colon X_{sq} & \to \bfn\\
                  t & \mapsto j_t.
	\end{align*}
	Then defining $\wt\left( J_{sq} \right)$ by 
	$$
	\wt\left( J_{sq} \right)_v = \# \left\{\, t\in X_{sq} \,\middle|\,
	J_{sq} = v
	\right\}, 
	$$
	we get 
	$\wt\left( J_{sq} \right) = \left( \theta_{s1q}, \dots,
	\theta_{snq} \right)$. On the other hand if we have a collection of
	functions $\left( J_{sq} \right)_{s,q=1}^n$ such that $\wt\left(
	J_{sq}
	\right) = \left( \theta_{s1q}, \dots, \theta_{snq} \right)$, then we can
	define $j\in I\left( n,r \right)$ by $j_t = J_{sq}\left( t
	\right)$ for $t\in X_{sq}$. Since $\bfr$ is the disjoint union of the sets
	$X_{sq}$ the multi-index $j$ is well defined. Moreover, 
	\begin{align*}
		\wt\left( i,j,k \right)_{stq} = \# \left\{\, 1\le v \le r \,\middle|\,
		i_v=s,\ j_v=t,\ k_v= q \right\} = \wt\left( J_{sq} \right)_t =
		\theta_{stq}.
	\end{align*}
	Thus $\left( i,j,k \right)\in \theta$. This gives a one-to-one
	correspondence between the set of those $j\in I\left( n,r \right)$ such that $\left(
	i,j,k
	\right)\in \theta$ and the set of collections of functions $\left( J_{sq}
	\right)_{s,q=1}^n$, $J_{sq}\colon X_{sq}\to \bfn$ such that $\wt\left(
	J_{sq} \right) = \left( \theta_{s1q},\dots, \theta_{snq} \right)$ for
	all $1\le s,q\le n$. 
	The number of possible choices for
	each $J_{sq}$ is 
	$$
	\binom{\theta_{s1q}+ \dots + \theta_{snq}}{\theta_{s1q},\dots,
	\theta_{snq}} = \binom{\left( \theta^2 \right)_{sq}}{\theta_{s1q},\dots,
	\theta_{snq}}.
	$$
	Since the choices of $J_{sq}$ for different pairs $\left( s,q
	\right)$ can be done independently we get that the number of elements in
	$\left\{\, j \,\middle|\, \left( i,j,k \right)\in\theta \right\}$ is
	given by $\left[ \theta \right]$. 
\end{proof}
\begin{remark}
	\label{prodzero}
	Given $\omega$, $\pi\in \Lambda\left( n,n;r \right)$
	 if there is $\theta\in \Lambda\left( n,n,n;r \right)$ such
	that $\theta^3 = \omega$ and $\theta^1 =\pi$ then
	$$
	\omega^1 = \left( \theta^3 \right)^1 = \left( \theta^1 \right)^2 =
	\pi^2.
	$$
	Thus if $\omega^1\not= \pi^2$ then $\xi_{\omega}\xi_{\pi} = 0$. 
	From this and the definition of $\xi_{\lambda}$ it follows that
	\begin{align*}
		\xi_{\lambda}\xi_{\omega} & = 
		\begin{cases}
			\xi_{\omega}, & \lambda = \omega^2\\
			0 , & \mbox{otherwise}
		\end{cases} & 
		\xi_{\omega}\xi_{\lambda} & = 
		\begin{cases}
			\xi_{\omega}, & \lambda = \omega^1 \\
			0, & \mbox{otherwise.}
		\end{cases}
	\end{align*}
\end{remark}

\section{The reduced bar construction}
	\label{bar_constr}
In this section we recall the construction of the  reduced bar resolution. This
is a
partial case of the construction described in Chapter IX, \S 7 of \cite{homology}.
We will use in this section a slightly different notation, namely we write
$i$, $j$, $k$ for natural numbers. 

Let $A$ be a ring with identity $e$ and $S$ a subring of $A$. We assume that in the category of rings there is a
splitting $p\colon A\to S$ of the inclusion map
$ S\to A$. We denote the kernel of $p$ by $I$. Then $I$ is an
$S$-bimodule. 
Denote by $\widetilde{p}$ the map from $A$ to $I$ given by $a\mapsto a-p\left(
a \right)$. Obviously $\widetilde{p}$ is a homomorphism of $S$-bimodules
and the restriction of $\widetilde{p}$ to $I$ is the  identity map.

For every left $A$-module $M$ we define the complex $B_k\left(
A,S,M \right)$, $k\ge -1$, as follows. We set $B_{-1}\left(
A, S,M \right)= M$, $B_0\left(
A, S,M \right) = A\otimes M$ and for $k\ge 1$, $B_k\left(A, S,M \right) = A\otimes
I^{\otimes k} \otimes M$, where all the tensor products are taken over $S$.
Next we define $A$-module homomorphisms $d_{kj}\colon B_{k}\left( A,S,M
\right)\to B_{k-1}\left( A,S,M \right) $, $0\le j\le k$, 
and $S$-module homomorphisms $s_k\colon B_k\left( A,S,M \right)\to B_{k+1}\left(
A,S,M \right)$ by
\begin{align*}
	d_{0,0} \left( a\otimes m \right) & := am\\
	d_{k,0} \left( a\otimes a_1 \otimes \dots \otimes a_k\otimes m \right)
	&:= a a_1 \otimes a_2 \otimes \dots \otimes a_k \otimes
	m\\
	d_{k,j} \left( a\otimes a_1 \otimes \dots \otimes a_k\otimes m \right)
	& :=  a\otimes a_1 \otimes \dots a_j a_{j+1} \otimes \dots a_k \otimes m
	& 1\le j\le k-1\\
	d_{k,k} \left( a\otimes a_1 \otimes \dots \otimes a_k\otimes m \right)
	&:= a\otimes a_1\otimes \dots \otimes a_{k-1} \otimes a_k m\\
	s_{-1} \left( m \right) & := e\otimes m\\
	s_k \left( a\otimes a_1 \otimes \dots \otimes a_k \otimes m \right) &:= 
	e \otimes  \widetilde{ p}\left( a \right)  \otimes a_1 \otimes
	\dots \otimes a_k \otimes m & 0 \le k. 
\end{align*}
It is easy to see
that 
\begin{align*}
	d_{k,i}d_{k+1,j} &= d_{k,j-1}d_{k+1,i} & 0\le i<j\le k\\
	s_{k-1} d_{k,j} &= d_{k+1,j+1} s_k & 1\le j\le k.
\end{align*}
Moreover for $k\ge 0$
\begin{align*}
	d_{0,0}s_{-1} \left( m \right) &= d_{0,0}\left( e\otimes m \right) = m\\
	d_{k+1,0}s_{k} \left( a\otimes a_1\otimes\dots \otimes a_k \otimes m \right) &=
	\widetilde{p}\left( a \right) \otimes a_1 \otimes \dots\otimes a_k \otimes m\\
d_{k+1,1}s_{k} \left( a\otimes a_1\otimes\dots \otimes a_k \otimes m \right) &=
e \otimes	\widetilde{p}\left( a \right) a_1 \otimes \dots\otimes a_k \otimes m
\\
s_{-1} d_{0,0}\left( a\otimes m \right) &= s_{-1}\left( am \right) = e\otimes
am\\
s_{k-1} d_{k,0}\left( a\otimes a_1\otimes \dots \otimes a_k \otimes m \right) &= 
e \otimes aa_1 \otimes a_2 \otimes\dots \otimes a_k \otimes m.
\end{align*}
Note that in the last identity we used the fact that $aa_1\in I$ which implies $\widetilde{p}\left(
aa_1 \right) = aa_1$. 
From the formulas above we have $d_{0,0}s_{-1} = \id_M$ and, taking into account that
$\widetilde{p}\left( a \right) = a- p\left( a \right)$ and $p\left( a \right)\in S$, 
\begin{align*}
	&	\left(	d_{k+1,0}s_k - d_{k+1,1}s_k + s_{k-1}d_{k,0} \right) \left(
	a\otimes a_1\otimes \dots \otimes a_k\otimes m \right) = \\
	&
	\phantom{\mbox{ab}} = 
	a\otimes a_1 \otimes \dots \otimes a_k \otimes m - e \otimes p\left(
	a \right)a_1 \otimes \dots \otimes a_k\otimes m\\
	&\phantom{ab=} - e\otimes aa_1 \otimes a_2 \otimes \dots \otimes a_k \otimes
	m + e\otimes p\left( a \right)a_1 \otimes a_2 \otimes \dots \otimes a_k \otimes
	m\\
	&\phantom{ab=} + e\otimes aa_1 \otimes a_2 \dots \otimes a_k \otimes m\\
	&\phantom{ab} = a\otimes a_1 \otimes \dots \otimes a_k \otimes m.
\end{align*}
Thus $d_{k+1,0}s_k - d_{k+1,1}s_k + s_{k-1}d_{k,0}  = \id_{B_k\left( A,S,M
\right)}$, $k\ge 0$.

Define $d_k\colon   B_k\left( A,S,M \right)\to
B_{k-1}\left( A,S,M \right)$ by 
\begin{align*}
	d_k&  := \sum_{t=0}^k \left( -1 \right)^t d_{k,t}.
\end{align*}
	The above computations show that 
\begin{proposition}
	\label{bar} The sequence $\left( B_k\left( A,S,M \right), d_k
	\right)_{k\ge -1}$ is a complex of left $A$-modules. Moreover 
	\begin{align*}
		d_{0}s_{-1}& = \id_{B_{-1}\left( A,S,
		M\right)}\\
		d_{k+1}s_k + s_{k-1} d_{k}& = \id_{B_k\left( A,S,M \right)} &
	 	0\le k.
	\end{align*}
	Thus $s_k$, $k\ge -1$, 
	give
	a splitting of $B_*\left( A,S,M \right)$ in the category of $S$-modules. In particular, $\left(
	B_k\left( A,S,M \right),d_k
	\right)_{k\ge -1}$ is exact.
\end{proposition}
\begin{definition}
	The complex $\left( B_k\left( A,S,M \right), d_k \right)_{k\ge -1}$ is
	called
	the
	\emph{reduced bar resolution}. 
\end{definition}
	\section{Borel-Schur algebras}
	\label{borel}
In this section we introduce the Borel-Schur algebra $S^+_R\left( n,r
\right)$ and apply the results of the previous section to the construction of
the reduced bar resolution for the irreducible $S^+_R\left( n,r
\right)$-modules. 
We should 
remark
that $S^+_R\left( n,r \right)$ can be identified with the Borel subalgebra 
$S\left(B^+  \right)$ defined in~\cite[ \S8 ]{greencombinatorics}. In case 
$R$ is an infinite field the algebra $S^+_R\left( n,r \right)$ can be also identified with
the algebra $S\left( B^+ \right)$ used
in~\cite{green2},~\cite{aps}, and~\cite{woodcock}. 

On the set $I\left( n,r \right)$ we define the ordering $\le$ by
$$
i\le j \Leftrightarrow i_1\le j_1,\ i_2\le j_2,\ \dots,\ i_n\le j_n.
$$
We write $i<j$ if $i\le j$ and $i\not= j$. 
Note that $\le$ is $\Sigma_r$ invariant.
Denote by $\Lambda^s\left( n,n;r \right)$ the set 
$$\Lambda^s\left( n,n;r \right) = \left\{\, \omega\in
\Lambda\left( n,n;r \right) \,\middle|\,
\parbox{11em}{$\omega$ is upper triangular and $$\sum_{1\le k\le l\le n} \left( l-k
\right)\omega_{kl}\ge s$$}  \right\}.$$ 
\begin{remark}
	\label{less}
Clearly $i\le j$ and $i<j$ are equivalent to $\wt\left( i,j \right)\in \Lambda^0\left( n,n;r
\right)$ and $\wt\left( i,j \right)\in \Lambda^1\left(n,n;r  \right)$,
respectively.
\end{remark}
	For $s\ge0$ denote by $J_s^R\left( n,r \right)$ the $R$-submodule of $S_R\left(
n,r \right)$ spanned by the set $\left\{\, \xi_{\omega} \,\middle|\, \omega\in
\Lambda^s\left( n,n;r \right)\right\}$. Note that for $s>r(n-1)$ the set $\Lambda^s\left( n,n;r \right)$ is empty,
therefore $J_s^R\left( n,r \right)$ is zero for $s\gg 0$.
\begin{proposition}
	\label{ideal}
	If $\omega\in \Lambda^s\left( n,n;r \right)$ and $\pi\in \Lambda^t\left(
	n,n;r \right)$, then $\xi_{\omega}\xi_{\pi}\in J_{s+t}^R\left( n,r
	\right)$. 
\end{proposition} 
\begin{proof}
	By Proposition~\ref{mult} $\xi_{\omega}\xi_{\pi}$ is a linear combination
	of the elements $\xi_{\theta^2}$ for $\theta\in \Lambda\left( n,n,n;r
	\right)$ satisfying $\theta^3=\omega$ and $\theta^1=\pi$.
	For any such $\theta$ the condition $\omega_{kq} =0$ for $k>q$ implies
	that $\theta_{kql}= 0$ for $k>q$, and the condition $\pi_{ql}=0$ for $q>l$ implies
	$\theta_{kql}=$ for $q>l$. Therefore for any $q$ and for $k>l$ we have
	$\theta_{kql}=0$. Hence $\theta^2$ is upper triangular. Moreover,
	\begin{align*}
		\sum_{k\le l} \left( l-k
		\right)\left( \theta^2 \right)_{kl} &= \sum_{k\le l}
		\sum_{q=1}^n \left( l-k \right)\theta_{kql} = \sum_{k\le q\le l} \left( l-k \right)
\theta_{kql}		\\& =\sum_{k\le q \le l} \left( l-q \right) \theta_{kql}+
		\sum_{k\le q\le l} \left( q-k \right)\theta_{kql} 
		\\& = \sum_{q\le l}\left( \sum_{k=1}^n \left( l-q \right)\theta_{kql} \right) + \sum_{k\le q}\left( \sum_{l=1}^n \left( q-k \right)\theta_{kql}
		\right)
		\\& = \sum_{q\le l}\left( l-q \right)\left( \theta^1 \right)_{ql}
		+ \sum_{k\le q}\left( q-k \right)\left( \theta^3 \right)_{kq}\\&  =
		\sum_{q\le l}\left( l-q \right)\pi_{ql}
		+ \sum_{k\le
		q}\left( q-k \right) \omega_{kq}\ge t + s. 
	\end{align*}
\end{proof}
Define  $S^+_R\left( n,r
\right) := J_0^R\left( n,r \right)$ and $J_R = J_R\left( n,r \right) := J_1^R\left(
n,r \right)$.
Then for every $\lambda\in \Lambda\left( n;r \right)$ we have
$\xi_{\lambda}\in S^+_R\left( n,r \right)$. In particular the identity of
$S_R\left( n,r \right)$ is in $S^+_R\left(
n,r \right)$. Now it follows from Proposition~\ref{ideal} that $S^+_R\left( n,r
\right)$ is a subalgebra of $S_R\left( n,r \right)$ and $J_R\left( n,r
\right)$ is a nilpotent ideal of $S^+_R\left( n,r
\right)$. 
\begin{definition}
The algebra $S^+_R\left( n,r
\right)$ is called the 
\emph{Borel-Schur algebra}. 
\end{definition}
Let $L_{n,r} = \bigoplus_{\lambda\in \Lambda\left( n,r \right)} R\xi_{\lambda}$.
Then $L_{n,r}$ is a commutative $R$-subalgebra of $S^+_R\left( n,r \right)$, and
 $S^+\left( n,r \right)= L_{n,r}\oplus J_R\left( n,r
\right)$. 
Since $J_R\left( n,r \right)$ is an ideal in $S^+_R\left( n,r \right)$,
this direct sum decomposition implies that the natural inclusion of
$L_{n,r}$ has a splitting in the category of $R$-algebras. So we can apply to
$S^+_R\left( n,r \right)$ and $L_{n,r}$ the reduced bar construction described
in Section~\ref{bar_constr}. 

For every $\lambda\in \Lambda\left( n;r \right)$ we have a rank one
module $R_\lambda:= R\xi_\lambda$ over $L_{n,r}$. Note that $\xi_{\lambda}$ acts on
$R_{\lambda} = R\xi_\lambda$ by identity, and $\xi_\mu$, $\mu\not=\lambda$, acts by zero. We
will denote in the same way the module over $S^+_R\left( n,r \right)$ obtained
from $R_\lambda$ by inflating along the natural projection of $S^+_R\left(
n,r \right)$ on $L_{n,r}$. 

Note that 	if $R$ is a field the algebra $L_{n,r}$ is semi-simple,
	and so $J_R\left( n,r \right)$ is the radical of $S^+_R\left(
	n,r\right).$ In this case $\left\{\, R_{\lambda} \,\middle|\,
	\lambda\in \Lambda\left( n,r \right) \right\}$ is a complete set of
	pairwise non isomorphic simple modules over $S^+_R\left( n,r \right)$.
	For more details the reader is referred to \cite{aps}. 

For $\lambda\in \Lambda\left( n;r \right)$ we denote the resolution $B_*\left(
S^+\left( n,r \right), L_{n,r}, R_\lambda
\right)$ defined in Section~\ref{bar_constr} by $B^+_*\left( R_\lambda \right)$. 
Then
\begin{align}
	\label{barplus}
B_k^+\left( R_\lambda \right) := S^+_R\left( n,r \right)\otimes J_R\left(
n,r\right)\otimes \dots \otimes J_R\left( n,r \right)\otimes R_{\lambda}
, 
\end{align}
	where all tensor products are over $L_{n,r}$ and there are $k$ factors $J_R\left( n,r
\right)$.

 Let $M$ be a right $L_{n,r}$-module and $N$ a left 
	$L_{n,r}$-module. It follows from Corollary~9.3 in~\cite{homology} that $M\otimes_{L_{n,r}} N = \bigoplus_{\lambda\in
	\Lambda\left( n;r \right)} M\xi_{\lambda}\otimes_{R
	\xi_{\lambda} } \xi_{\lambda} N$. 
	Hence $M\otimes_{L_{n,r}}N \cong \bigoplus_{\lambda\in \Lambda\left(
	n;r \right)} M\xi_{\lambda}\otimes_R \xi_{\lambda}N $.

Therefore $B_k^+\left( R_\lambda \right)$ is the
direct sum over all sequences $\mu^{\left( 1 \right)}$, \dots, $\mu^{\left( k+1
\right)}\in \Lambda\left(n;r 
\right)$ of the $S^+\left( n,r \right)$-modules
\begin{align}
	\label{resolutionplus}
	S^+_R\left( n,r \right)\xi_{\mu^{\left( 1 \right)}}\otimes
	\xi_{\mu^{\left( 1
	\right)}}J_R\left( n,r \right)\xi_{\mu^{\left( 2 \right)}}\otimes \dots
\otimes
\xi_{\mu^{\left( k \right)}}J_R\left( n,r \right)\xi_{\mu^{\left( k+1
\right)}}\otimes \xi_{\mu^{\left( k+1 \right)}}R_{\lambda},
\end{align}
	where all tensor products are over $R$. Since $\xi_{\mu^{\left( k+1 \right)}} R_{\lambda} = 0$
unless $\mu^{(k+1)}= \lambda$, the summation is in fact over the sequences
$\mu^{\left( 1 \right)}$,
\dots, $\mu^{\left( k \right)}\in \Lambda\left( n,r \right)$. 

Recall that it is said that   $\nu\in \Lambda\left( n;r \right)$
\emph{dominates} $\mu\in\Lambda(n;r)$ 
if 
\[
\sum\limits_{s=1}^t\nu_s\ge\sum\limits_{s=1}^t\mu_s
\]
for all $t$. If $\nu$ dominates $\mu$ we write $\nu\dominate\mu$.  
If $\nu$ strictly dominates $\mu$ we write $\nu\strdominate\mu$. 
Note that if $i$, $j\in I\left( n,r \right)$ and $i\le j$ or $i<j$, then $\wt\left( i
\right)\dominate \wt\left( j \right)$ or $\wt\left( i \right)\strdominate
\wt\left( j \right)$, respectively. 
\begin{proposition}
	\label{dominate} Let $\nu$, $\mu\in \Lambda\left( n;r \right)$. Then
	$\xi_{\nu} J_R\left( n,r \right)\xi_{\mu} = 0$ unless
	$\nu\strdominate\mu$. If $\nu\strdominate\mu$ then $$
	\left\{\, \xi_{\omega} \,\middle|\,\omega\in \Lambda^1\left(n,n;r 
	\right), \ \omega^1 = \mu,\ \omega^2= \nu   \right\}
	$$
is an $R$-basis of the free $R$-module $\xi_{\nu} J_R\left(
	n,r \right)\xi_{\mu}$.
\end{proposition}
\begin{proof}
	From   Remark~\ref{prodzero} it follows that the set
	$$
	\left\{\, \xi_{\omega} \,\middle|\,\omega\in \Lambda^1\left(n,n;r 
	\right), \ \omega^1 = \mu,\ \omega^2= \nu   \right\}
	$$
	is an $R$-basis for $\xi_{\nu}J_R\left( n,r \right)\xi_\mu$. Suppose
	this
	is not empty.  
	Then there is $\omega\in \Lambda^1\left( n,n;r \right)$ with $\omega^1 =
	\mu$ and $\omega^2 = \nu$. 
	Let $\left( i,j \right)\in \omega$. Then by Remark~\ref{less} we know
	that $i<j$. Therefore,  we have $\wt\left( i
	\right)\strdominate \wt\left( j \right)$. By Remark~\ref{commutative} we
	get
	$$
	\nu = \omega^2 = \wt\left( i,j \right)^2 = \wt\left( i \right)
	\strdominate \wt\left( j \right) = \wt\left( i,j \right)^1 = \omega^1 =
	\mu.
	$$
\end{proof}
\begin{corollary}
	\label{finite}
	Let $N$ be the length of the maximal strictly decreasing sequence in
	$\left( \Lambda\left( n;r \right), \strdominate \right)$. Then 
 $B^+_k\left(R_{\lambda}  \right) =0 $ for	$k>N$. 
\end{corollary}
We conclude that the resolutions $B^+_*\left(R_\lambda  \right)$ are finite, for
all $\lambda\in \Lambda\left( n;r \right)$. We also have $B_0^+\left( R_\lambda
\right) \cong S^+_R\left( n,r \right)\xi_{\lambda}$, and for $k\ge 1$
\begin{align*}
	B^+_k\left( R_\lambda \right) = \bigoplus_{\mu^{\left( 1
	\right)}\strdominate \dots \strdominate \mu^{\left( k
	\right)}\strdominate \lambda } S^+_R\left( n,r
	\right)\xi_{\mu^{\left( 1 \right)}} \otimes_R \xi_{\mu^{\left( 1
	\right)}} J_R \xi_{\mu^{\left( 2 \right)}} \otimes_R \dots \otimes_R
	\xi_{\mu^{\left( k \right)}} J_R \xi_{\lambda}, 
\end{align*}
where all $\mu^{\left( s \right)}$, $1\le s \le k$, are elements of $\Lambda\left( n;r
\right)$. 
Given $\mu\strdominate \lambda$ we denote by $\Omega_k^+\left( \lambda,\mu
\right)$ the set 
$$
\left\{\, \left( \omega_1,\dots, \omega_k
\right)\in \Lambda^1\left(n,n;r  \right)^k \,\middle|\,\substack{ \left( \omega_1
\right)^2 = \mu, \left( \omega_k
\right)^1 = \lambda
\\[1ex] \left( \omega_1 \right)^1 = \left( \omega_2 \right)^2,\dots,
\left( \omega_{k-1} \right)^1 = \left( \omega_k \right)^2}\right\}.
$$
Then 
\begin{align}
	\nonumber
	B_0^+ \left( R_\lambda \right) & \cong S^+_R\left( n,r
	\right)\xi_\lambda \\
	\label{eq:bk}
	B_k^+\left( R_\lambda \right)& \cong \bigoplus_{\mu\strdominate \lambda }
\left(
S^+_R\left( n,r \right)\xi_{\mu} \right)^{ \#\Omega^+_k\left(\lambda, \mu
\right)}, & k\ge 1 
\end{align}
	as  $S^+_R\left( n,r \right)$-modules. 
In particular, $B^+_k\left( R_\lambda \right)$ is a projective $S^+_R\left( n,r
\right)$-module for any $k\ge 0$. 
So we have the following result
\begin{proposition}
	\label{bplusresolution}
	Let $\lambda\in \Lambda\left( n;r \right)$. Then $B^+\left( R_\lambda
	\right)$ is a projective resolution of the module $R_\lambda$
	over $S^+_R\left( n,r \right)$.
\end{proposition}
Define
\begin{align*}
b_k\left( \lambda \right) := \left\{\, \left( \omega_0,\omega_1,\dots, \omega_k
\right) \,\middle|\,
\omega_0\in \Lambda^0\left( n,n;r \right),\ \left( \omega_1,\dots,\omega_k
\right)\in \Omega^+_k\left(\lambda, \left( \omega_0 \right)^1\right) \right\}. 
\end{align*}
	Then the set 
$$
\left\{\, \xi_{\omega_0}\otimes \dots \otimes \xi_{\omega_k} \,\middle|\, \left(
\omega_j \right)_{j=0}^k\in b_k\left( \lambda \right)\right\}
$$
is an $R$-basis of $B_k^+\left( R_\lambda \right)$. The differential $\partial$ of
$B^+_*\left( R_\lambda \right)$ in terms of this basis looks like
\begin{align}
	\label{eq:diff}
\partial_k	\left(  \xi_{\omega_0}\otimes \dots \otimes \xi_{\omega_k}
\right) = \sum_{t=0}^{k-1} \left( -1 \right)^t  \xi_{\omega_0}\otimes \dots
\otimes\xi_{\omega_t}\xi_{\omega_{t+1}}\otimes\dots\otimes \xi_{\omega_k}.
\end{align}
Let $R'$ be another commutative ring with identity and $\phi\colon R\to R'$ a homomorphism of
rings. It is now clear that
\begin{align*}
	B^+_*\left( R_\lambda \right)\otimes_R R' \cong B^+_*\left( {R'}_\lambda \right)
\end{align*}
as chain complexes of $R'$-modules. 

\section{Induction and Woodcock's theorem}
\label{induction}
In this section we explain how the results of Woodcock~\cite{woodcock} can be
applied to prove that the $S^+_R\left( n,r \right)$-module $R_\lambda$, $\lambda\in
\Lambda^+\left( n,r \right)$, is acyclic for the induction functor $S^+_R\left(
n,r \right)\lmod\to S_R\left( n,r \right)\lmod$. 
We start with a definition.

\begin{definition}
For $\lambda\in \Lambda^+\left( n,r \right)$ the $S_R\left( n,r \right)$-module
$W^R_\lambda := S_R\left( n,r \right)\otimes_{S^+_R\left( n,r \right)}R_\lambda$ is called
the \emph{Weyl module} for $S_R\left( n,r \right)$ associated with $\lambda$. 
\end{definition}
 It follows from Theorem~8.1~in~\cite{greencombinatorics}
that this definition is equivalent to the definition of Weyl module 
given in that work. It is proved in Theorem~7.1(ii) of \cite{greencombinatorics}
that  $W^R_\lambda$ is a free  $R$-module. In fact in this theorem
an explicit description of an $R$-basis for $W^R_\lambda$ is given. It turns out
that this basis does not depend on the ring $R$ but only on the partition
$\lambda$. Moreover the coefficients in the formulas for the action of $S_R\left( n,r \right)$ on
$W^R_\lambda$ are in  $\Z$ and do not depend on $R$. This
implies that if $\phi\colon R\to R'$ is a ring homomorphism then 
there is an isomorphism of $S_{R'}\left( n,r \right)$-modules
\begin{align}
	\label{eq:baseweyl}
	W^{R'}_{\lambda}\cong  W^R_{\lambda} \otimes_R R'.
\end{align}
For each $\lambda\in \Lambda^+\left( n;r \right)$ define the complex 
$$
B_*\left(W^R_\lambda \right) := S_R\left( n,r
\right)\otimes_{S^+_R\left( n,r \right)}
B^+_*\left( R_\lambda \right).
$$
Note that $B_{-1}\left( W^R_\lambda \right)= W^R_\lambda$. 
Now,
for any $\mu\in \Lambda\left( n;r
\right)$
\begin{align*}
S_R\left( n,r \right)\otimes_{S^+_R\left(n,r  \right) }S^+_R\left( n,r
\right) \xi_{\mu}   \cong S_R\left( n,r \right)\xi_\mu
\end{align*}
	is a projective $S_R\left( n,r \right)$-module.
	From \eqref{eq:bk} it follows that for $k\ge 1$ the $S_R\left( n,r
	\right)$-module $B_k\left( W^R_\lambda \right)$ is
	isomorphic to a direct sum of modules $S_R\left( n,r \right)\xi_\mu$,
	$\mu\strdominate \lambda$.
	Also $B_0\left( W^R_\lambda \right)\cong S_R\left( n,r
	\right)\xi_{\lambda}$. 
Thus $B_k\left( W^R_\lambda \right)$ is a projective $S_R\left( n,r \right)$-module for $k\ge 0$.

	Denote by $\widetilde{b^k}\left( \lambda \right)$ the set
	$$
	\left\{\, \left( \omega_0,\omega_1,\dots, \omega_k \right) \,\middle|\,
	\omega_0\in \Lambda\left( n,n;r \right),\ \left( \omega_1,\dots,\omega_k
	\right)\in \Omega^+_k\left( \lambda,\left(  \omega_0 \right)^1
	\right)\right\}.
	$$
	Then $B_0\left( W^R_\lambda \right)$ has an $R$-basis $\left\{\,
	\xi_{\omega}
	\,\middle|\,  \omega\in \Lambda\left( n;r \right),\ \omega^1 = \lambda
	\right\}$. 
	For $k\ge 1$ 
	\begin{align*}
	B_k\left(W^R_\lambda \right) = \bigoplus_{\substack{
	\mu^{\left( 1 \right)}\strdominate \dots \strdominate \mu^{\left( k \right)}\strdominate \lambda
	\\[1ex]
	\mu^{\left( s \right)}\in \Lambda\left( n;r \right),\ 1\le s\le k}
	} S_R\left( n,r
	\right)\xi_{\mu^{\left( 1 \right)}} \otimes_R \xi_{\mu^{\left( 1
	\right)}} J_R \xi_{\mu^{\left( 2 \right)}} \otimes_R \dots \otimes_R
	\xi_{\mu^{\left( k \right)}} J_R \xi_{\lambda}
\end{align*}
	 has an $R$-basis 
	$$
	\left\{\, \xi_{\omega_0}\otimes \xi_{\omega_1}\otimes \dots\otimes
	\xi_{\omega_k} \,\middle|\, \left( \omega_0,\dots,\omega_k \right)\in
	\widetilde{b_k}\left( \lambda \right) \right\}.
	$$
The differential $\partial$ of  $B_*\left( W^R_\lambda \right)$ in the terms of these bases is given
again by \eqref{eq:diff}. 	
From these facts it follows that, if $\phi\colon R'\to R$ is a homomorphism of
rings, then 
\begin{align}
	\label{eq:basechange}
	B_*\left( W^R_\lambda \right)\otimes_R R' \cong B_*\left( W^{R'}_\lambda
	\right). 
\end{align}
\begin{theorem}
	\label{woodcock}
	Let $\lambda\in \Lambda^+\left( n;r \right)$. 
	The complex  $B_*\left( W^R_\lambda \right)$ is a projective resolution of $W^R_\lambda$ over $S_R\left( n,r
	\right)$. 
\end{theorem}
\begin{proof}
	Fix $\lambda\in \Lambda^+\left( n;r \right)$ and denote the complex
	$B_*\left( W^R_\lambda \right)$ by $\X\left( R
	\right)$. Then all $R$-modules in $\X\left( R \right)$ are free
	$R$-modules.
	Moreover we know  
	 from \eqref{eq:basechange} that
	$\X\left( Z \right)\otimes_{\Z} R \cong \X\left( R \right)$. Now by the
	Universal Coefficient Theorem
	(cf. for example Theorem~8.22 in \cite{rotman}) we have a short exact sequence
	\begin{equation}	
		\label{eq:uct}
		0 \to H_k\left( \X\left( \Z \right) \right) \otimes_{\Z} R \to
		H_k\left(
		\X\left( R \right) \right) \to \Tor^{\Z}_1\left( H_{k-1}\left(
		\X\left(
		Z \right)
		\right), R \right)\to 0.
	\end{equation}
		Thus to show that the complexes $\X\left( R \right)$ are acyclic it is
	enough to check that the complex $\X\left( \Z \right)$ is acyclic. 
	Now $H_k\left( \X\left( \Z \right) \right)$ is a finitely generated abelian group.
	Therefore 
	$$
	H_k \left( \X\left( \Z \right) \right) \cong \Z^t \oplus
	\bigoplus_{p\mbox{ prime}} \bigoplus_{s\ge 1} \left(
	\left.\raisebox{0.3ex}{ $\Z$}\middle/\!\raisebox{-0.3ex}{ $p^s\Z$}\right.
	\right)^{t_{ps}},
	$$
	where only finitely many of the integers $t$, $t_{ps}$ are different from
	zero. For every prime $p$ denote by $\overline{ \F }_p$ the algebraic
	closure of $\F_p$. Then we get 
	$$
	H_k\left( \X\left( \Z \right) \right) \otimes_{\Z} \overline{ \F }_p \cong
	\overline{ \F }_p^{\sum_{s\ge 1} t_{ps}}
	$$
	and also
	$$
	H_k\left( \X\left( \Z \right) \right)\otimes_{\Z} \Q \cong \Q^t.
	$$
	Hence if we show that $H_k\left( \X\left( \Z \right)
	\right)\otimes_{\Z} \overline{ {\F} }_p = 0$ for all prime numbers
	$p$ and $H_k\left( \X\left( \Z \right) \right) \otimes_{\Z} \Q = 0$,
	this
	 will imply that $H_k\left( \X\left( \Z \right) \right) = 0$. 

	Let $\K$ be one of the fields $\overline{ \F }_p$, $p$ prime, or $\Q$. 
	Then, by the Universal Coefficient Theorem, $H_k\left( \X\left( \Z
	\right) \right)\otimes_{\Z} \K$ is a submodule of $H_k\left( \X\left(
	\K \right)
	\right)$. Therefore it is enough to show that $H_k\left( \X\left( \K
	\right) \right)=0 $. For this 
 we use Theorem~5.1~in~\cite{woodcock}. 
 That result can be applied because
$\K$ is an infinite field.

Note that the algebra $S_R\left( n,r \right)$ has an anti-involution
$\involution\colon
S_R\left( n,r \right)\to S_R\left( n,r \right)$ defined on the basis elements by
$\involution\colon \xi_{\omega}\mapsto \xi_{\omega^t}$. The image of $S^+_R\left( n,r
\right)$ under $\involution$ is the  subalgebra $S^-_R\left( n,r \right)$ of $S_R\left( n,r
\right)$. Now for each $S^+_\K\left( n,r \right)$-module $M$ we define a
structure of  
$\involution\left( S^+_\K\left( n,r \right) \right)$-module  on
$M^* := \Hom_\K\left( M,\K \right)$ by  $\left( \xi\theta \right)\left( m
\right) := \theta\left( \involution\left( \xi \right)m \right)$, for $\theta\in
M^* $, $\xi\in \involution\left( S^+_\K\left( n,r \right) \right)$. This induces
a contravariant equivalence of categories $\involution_*\colon S^+_\K\left(
n,r \right)\mbox{-mod} \to S^-_\K\left( n,r \right)\mbox{-mod}$.
There is also a similarly defined contravariant auto-equivalence functor
$\involution_*\colon S_\K\left( n,r \right)\mbox{-mod}\to S_\K\left( n,r
\right)\mbox{-mod}$. By Theorem~7.1~in~\cite{aps} the functors
\begin{align*}
\involution_*\circ
\Hom_{S^-_\K\left( n,r \right)} \left( S_\K\left( n,r \right), - \right)\circ
\involution_* \colon S^+_\K\left( n,r \right)\mbox{-mod} &
 \to S_\K\left( n,r \right)\mbox{-mod}
\\ S_\K\left( n,r \right)\otimes_{S^+_\K\left( n,r \right)} -\colon
S^+_\K\left( n,r \right)\mbox{-mod}&\to S_\K\left( n,r \right)\mbox{-mod}
\end{align*}
are isomorphic. Now since $\involution_*$ is a contravariant equivalence of
abelian categories, $\involution_*\left( B\left( \K_{\lambda}\right) \right)$
is an injective resolution of $\involution_*\left( \K_\lambda \right)$. By
Theorem~5.1~of~\cite{woodcock} the complex 
$$
\Hom_{S^-_\K\left( n,r \right)}\left(S_\K\left( n,r \right),  J_*\left( B_*\left(
{\K}_{\lambda} \right) \right) \right)
$$
is exact. Applying $\involution_*$ we get that also the complex 
$ \X\left( \K \right)$
is exact. 
\end{proof}
\begin{corollary}
	\label{woodcock2}
	Let $\lambda\in \Lambda^+\left( n;r \right)$ and $P\to R_\lambda$ be a projective resolution of $R_\lambda$ over
	$S^+_R\left( n,r \right)$. Then $S_R\left( n,r \right)\otimes_{S^+\left(
	n,r \right)} P\to W^R_\lambda$ is a projective resolution of
	$W^R_\lambda$ over $S_R\left( n,r \right)$. 
\end{corollary}
\begin{proof}
	The homology groups of the complex $S_R\left( n,r \right)\otimes_{S^+\left(
	n,r \right)} P\to 0$ are the tor-groups $Tor^{S^+\left( n,r \right)}_k
	\left( S_R\left( n,r \right), R_\lambda \right)$. But these groups are
	zero  for $k\ge 1$ by  Theorem~\ref{woodcock}. And in degree zero we get
	$W_\lambda^R$.
\end{proof}
From Theorem~\ref{isodivided} it follows that the resolution $B_*\left(
W^R_\lambda
\right)$  is a resolution of $W^R_\lambda$ by divided power
modules over $\GL_n\left( R \right)$. Moreover, this resolution is universal in
the sense of Akin and Buchsbaum. Namely, we have
$$
B_*\left( W^R_\lambda \right) = B_*\left( W^\Z_\lambda \right)\otimes_{\Z} R.
$$
So this resolution answers the question posed by Akin and Buchsbaum referred in
the introduction of the paper.

\section{The Schur algebra and  the symmetric group}
\label{symmetric}
In this section we recall the connection which exists between the categories of
representations of the symmetric group $\Sigma_r$ and of the Schur algebra
$S_R\left( n,r \right)$. Good references on this subject are~\cite{green}
and~\cite{james}. 

Suppose that $n\ge r$. Then there is  $\delta= \left( 1,\dots,1,0,\dots,0
\right)\in \Lambda\left( n;r \right)$.
For every $\sigma\in \Sigma_r$ we denote  by $\omega\left( \sigma \right)$ the 
element of $\Lambda\left( n,n;r \right)$ defined by
$$
\omega\left( \sigma \right)_{st} := 
\begin{cases}
	1, &  1\le t\le r,\ s=\sigma t\\
	0, & \mbox{otherwise.}
\end{cases}
$$
This gives a one-to-one correspondence between $\Sigma_r$ and the elements
$\omega$ of
$\Lambda\left( n,n;r \right)$ such that $\omega^1=\omega^2=\delta$.
It follows from the Proposition~\ref{mult} that the map 
\begin{align*}
	\Sigma_r & \to \xi_{\delta}S_R\left( n,r \right)\xi_{\delta}\\
	\sigma & \mapsto \xi_{\omega\left( \sigma \right)}
\end{align*}
is multiplicative. In fact, if $\theta\in \Lambda\left( n,n,n;r \right)$ is such
that $\theta^3=\omega\left( \sigma_1 \right)$ and $\theta^1= \omega\left(
\sigma_2
\right)$, then 
$$
\theta_{stq} = 
\begin{cases}
	1 , & 1\le q\le r,\ s=\sigma_1 t,\ t= \sigma_2 q\\
	0, & \mbox{otherwise}
\end{cases}.
$$
Such $\theta$ is unique and for this $\theta$ we have $\theta^2 = \omega\left( \sigma_1\sigma_2 \right)$ and
$\left[ \theta \right]=1$. 
Therefore $\xi_{\omega\left(\sigma_1\right)}\xi_{\omega\left(\sigma_2\right)}=
\xi_{\omega\left(\sigma_1\sigma_2\right)}$. We can
consider $\xi_{\delta}S_R\left( n,r \right)\xi_{\delta}$ as an algebra with
identity
$\xi_\delta$. Then $\sigma\mapsto \omega\left( \sigma \right)$ induces an isomorphism of algebras
$\phi\colon R \Sigma_r \stackrel{\cong}{\to} 
\xi_\delta S_R\left( n,r \right)\xi_{\delta}$. 
It is obvious that if $M$ is an $S_R\left( n,r \right)$-module then
$\xi_{\delta} M$ is an $\xi_\delta S_R\left( n,r \right)\xi_\delta $-module. In fact the map $M\mapsto \xi_\delta M$ is functorial. Now
we can consider $\xi_\delta M$ as $R\Sigma_r$-module via $\phi$.
The resulting functor $\sf\colon S_R\left( n,r \right)\lmod \to R\Sigma_r
\lmod$ was named \emph{Schur functor} in~\cite{green}. 
This terminology should not be mixed up with Schur functors in~\cite{abw2}. 
It is obvious that $\sf$ is exact.

Given $\lambda\in\Lambda^+\left( n;r \right)$, we can apply the functor
$\sf$ to the complex $B_*\left( W^R_\lambda \right)$. We obtain an exact
sequence $\sf \left( B_*\left( W^R_\lambda \right) \right)$, which is a
resolution by permutation modules of the co-Specht module corresponding to $\lambda$.
This will be explained in more detail in the next section. 

\section{The Boltje-Hartmann complex}
\label{boltjehartmann}
\label{boltje}
In this section we show that for $\lambda\in \Lambda^+\left( n;r \right)$ the
complex  $\sf\left( B_*\left(
W^R_\lambda
\right) \right)$ is isomorphic to the complex constructed in~\cite{boltje}.
Thus we prove Conjecture~3.4 of~\cite{boltje}. 

We start by summarizing notation and conventions of~\cite{boltje}.

Let $\lambda\in \Lambda\left( n;r \right)$. The
\emph{diagram} of shape $\lambda$ is the subset of $\N^2$
\begin{align*}
	\left[ \lambda \right] = \left\{\, \left( s,t \right) \,\middle|\, 1\le
	t\le \lambda_s,\ 1\le s\le n\right\}, 
\end{align*}
and a \emph{tableau} of shape $\lambda$ or \emph{$\lambda$-tableau} is a map
$T\colon \left[ \lambda \right]\to \N$. The content $c\left( T \right)$ of $T$ is defined by
\begin{align*}
	c(T)_t = \# \left\{\, \left( s,q \right)\in \left[ \lambda \right]
	\,\middle|\,  T\left( s,q \right) = t \right\}. 
\end{align*}
If all values of $T$ are no greater than $n$ we consider $c\left( T
\right)$ as an element of $\Lambda\left( n;r \right)$. 
For every pair
$\lambda$, $\mu\in \Lambda\left( n;r \right)$, denote by $\tableau\left(
\lambda,\mu
\right)$ the set of tableaux of shape $\lambda$ and content $\mu$.
We say that a tableau $T\in \tableau\left( \lambda,\mu \right)$ is row
semistandard if for every $1\le s \le n$
$$
T\left( s,1 \right)\le T\left( s,2 \right)\le \dots\le T\left( s,\lambda_s
\right). 
$$
The set of row semistandard tableaux of shape $\lambda$ and content $\mu$ will
be denoted by $\tableau^{rs}\left( \lambda,\mu \right)$. 
Following~\cite{boltje} we denote by $\tableau\left( \lambda \right)$
(respectively $\tableau^{rs}\left( \lambda \right)$) the
set $\tableau\left( \lambda,\delta \right)$ (respectively $\tableau^{rs}\left( \lambda,\delta
\right)$), where $\delta=\left(
1^r,0^{n-r}
\right)$ like in the previous section. Note that every element
$\mathfrak{t}$ of $\tableau\left( \lambda \right)$ is a bijection from
$\left[ \lambda \right]$ to $\bfr$. Therefore we can
define a left action of $\Sigma_r$ on $\tableau\left( \lambda \right)$ by
\begin{align*}
	\Sigma_r \times \tableau\left( \lambda \right) & \to \tableau\left(
	\lambda
	\right)\\
\left( \sigma,\mathfrak{t} \right)& \mapsto \sigma
\mathfrak{t}.
\end{align*}
Define the projection $\rs\colon \tableau\left( \lambda \right)\to
\tableau^{rs}\left( \lambda \right)$ that turns $\mathfrak{t}\in
\tableau\left( \lambda\right)$ into the row semistandard tableau obtained
from $\mathfrak{t}$ by 
rearranging the elements of each row of $\mathfrak{t}$ in increasing order.
 Then 
\begin{align*}
	\Sigma_r\times \tableau^{rs}\left( \lambda \right) &\to
	\tableau^{rs}\left( \lambda \right)\\
	\left( \sigma,\mathfrak{t} \right) & \mapsto \rs\left(
	\sigma\mathfrak{t} \right)
\end{align*} 
is a left transitive action of $\Sigma_r$ on $\tableau^{rs}\left( \lambda
\right)$, and the stabilizer of \begin{align*}
	\mathfrak{t}_\lambda := 	\begin{array}{ccccccc}
		1 & \dots & \dots & \dots & \lambda_1\\
		\lambda_1 + 1 & \dots & \dots & \lambda_1 + \lambda_ 2\\
		\vdots & \vdots & \vdots \\
		\lambda_1 + \dots + \lambda_{n-1} + 1 & \dots & \lambda_1 +
		\dots + \lambda_n
	\end{array}
\end{align*}
is $\Sigma_\lambda = \Sigma_{\lambda_1}\times \dots \times \Sigma_{\lambda_n}$.
	Following~\cite{boltje}, we define the left permutational $\Sigma_r$- module $M^\lambda$
	 as the linear span over $R$ of the elements in
	$\tableau^{rs}\left( \lambda \right)$, with the action induced on
	$M^\lambda$ by the formula above. 

	Let $V_\lambda = \xi_\lambda\left( \left( R^n \right)^{\otimes r}
	\right) = \bigoplus_{i\in \lambda}Re_i$. Then $V_\lambda$ is a right $\Sigma_r$-submodule of $\left(
	R^n \right)^{\otimes r}$. We will consider $V_\lambda$ as a left
	$\Sigma_r$-module via the transitive $\Sigma_r$-action  $\sigma e_i =
	e_{i\sigma^{-1}}$ 
	on the basis $\left\{\, e_i \,\middle|\, i\in \lambda \right\}$. 
	 Obviously, $V_\lambda$ is a permutational module. 
	 Let $l\left( \lambda \right) := (1^{\lambda_1},\dots, n^{\lambda_n})$.
	Then the stabilizer of $e_{l\left( \lambda \right)}$ is
	$\Sigma_\lambda$, and so
	the $\Sigma_r$-modules $M^\lambda$ and $V_\lambda$ are
	isomorphic. 
	Let us describe the isomorphism induced by $e_{l\left( \lambda
	\right)}\to \mathfrak{t}_\lambda$ explicitly. 

	For every $i\in \lambda$, we define $\mathfrak{t}\left( i \right)\in
	\tableau^{rs}\left( \lambda \right)$ as the row semistandard
	$\lambda$-tableau whose row $s$ contains the $\lambda_s$ integers
	$q$ satisfying $i_q = s$, $s\in \bfn$.

	We have $\mathfrak{t}\left( l\left(
	\lambda \right)
	\right) = \mathfrak{t}_\lambda$.
	Next we 
	 check that the
	correspondence $e_{i}\mapsto \mathfrak{t}\left( i \right)$ is
	$\Sigma_r$-invariant.  

Let $i\in \lambda$, $\sigma\in \Sigma_r$ and $1\le s\le n$.  Suppose $s$ occurs
at positions $\nu_1 < \dots <\nu_{\lambda_s}$ in $i$. Then 
\begin{align*}
	s = i_{\nu_q}& = i_{\sigma^{-1}\left( \sigma\left( \nu_q \right)
	\right)}, & 1 \le q \le \lambda_s. 
\end{align*}
	Hence $s$ occurs at the positions $\sigma\left( \nu_1 \right)$, \dots,
	$\sigma\left( \nu_{\lambda_s} \right)$ in $i\sigma^{-1}$. 
	Since $s$ is arbitrary we get $\mathfrak{t}\left(i\sigma^{-1}  \right)
	=\rs\left(\sigma\mathfrak{t}\left( i \right)  \right) $. 
	Therefore the map $e_i\mapsto \mathfrak{t}\left( i \right)$ is
	$\Sigma_r$-invariant, and defines the $\Sigma_r$-isomorphism between
	$V_\lambda$ and $M^\lambda$ referred above.  

	As a consequence,
for $\lambda$, $\mu\in \Lambda\left( n;r \right)$ we get an
isomorphism $\Hom_{R\Sigma_r}\left( M^\mu,M^\lambda \right)\to
\Hom_{R\Sigma_r}\left( V^\mu,V^\lambda \right)$. Since $\left( R^n
\right)^{\otimes r} \cong \bigoplus_{\lambda\in \Lambda\left( n;r \right)}
V_\lambda$ these
isomorphisms can be assembled into an isomorphism of algebras
\begin{align}
	\label{eq:isom7}
\bigoplus_{\lambda,\mu\in\Lambda\left( n;r
\right)}\Hom_{R\Sigma_r}\left(M^\mu,M^\lambda
\right)\stackrel{\cong}{\longrightarrow} S_R\left( n,r \right),
\end{align}
where
we consider the direct sum on the left hand side as an algebra, with the product
of two composable maps given by their composition, and the product of two
non-composable maps is defined to be zero. 

We will describe \eqref{eq:isom7} explicitly below. 
Let
$$
\Omega\left( \lambda,\mu \right) =  \left\{\, \omega\in \Lambda\left( n,n;r
\right) \,\middle|\,  \omega^1 = \mu, \ \omega^2 = \lambda
\right\}.
$$
Then there is a bijection between the sets $\tableau^{rs}\left( \lambda,\mu \right)$
and $\Omega\left( \lambda,\mu \right)$. 
To prove this, we define for each  $T\in \tableau^{rs}\left( \lambda,\mu
\right)$  the matrix $\omega\left(
T \right)\in \Lambda\left( n,n;r \right)$ by
$$
\omega_{st} := \# \left\{\, \left( s,q \right)\in \left[ \lambda \right]
\,\middle|\, T\left( s,q \right) = t \right\} = \left(\mbox{number of $t$'s in row
$s$ of $T$}\right).
$$
Clearly $\omega\left( T \right)^2 = \lambda$ and $\omega\left( T \right)^1 =
\mu$. Thus $\omega\left( T \right) \in \Omega\left( \lambda,\mu \right)$.  
Conversely, if $\omega\in \Omega\left( \lambda, \mu \right)$ we define 
$T\left( \omega \right)$ as the tableau of shape $\lambda$ whose $s$-th row is the sequence 
$\left( 1^{\omega_{s1}}, \dots, n^{\omega_{sn}} \right)$. Then $T\left( \omega
\right)$ is row
semistandard. Moreover, $t$ occurs in
$T$ exactly $\left( \omega^1 \right)_t= \mu_t$ times, and so $T\left( \omega
\right)$ has content
$\mu$. 
It is easy to see that these constructions are mutually inverse.

For every $T\in \tableau^{rs}\left( \lambda,\mu \right)$ Boltje and Hartmann
define a map $\theta_T\colon M^\mu\to M^\lambda $ by the rule:
for $\mathfrak{t}\in \tableau^{rs}\left( \mu \right)$ the element
$\theta_{T}{\mathfrak{t}}\in M^\lambda $ is equal to the sum of all
$\lambda$-tableaux $\mathfrak{s}\in \tableau^{rs}\left( \lambda \right)$ with
the following property, for each $1\le s\le n$: if the $s$-th row of $T$ contains
precisely $q$ entries equal to $t$ then the $s$-th row of $\mathfrak{s}$
contains precisely $q$ entries from the $t$-th row of $\mathfrak{t}$.

We can reformulate this rule in terms of $\omega\left( T \right)$ as follows:
the element $\theta_{T} \mathfrak{t} $ is the sum of those $\mathfrak{s}\in
\tableau^{rs}\left( \lambda \right)$
that, for each $1\le s,t\le n$, the $s$-th row of $\mathfrak{s}$ contains
precisely $\omega\left( T \right)_{st}$ entries from the $t$-th row of
$\mathfrak{t}$. 

The set $\left\{\, \theta_T \,\middle|\, T\in \tableau^{rs}\left( \lambda,\mu
\right) \right\}$ is an  $R$-basis of $\Hom_{R\Sigma_r}\left( M^\mu,M^\lambda
\right)$. On the other hand $\left\{\, \xi_{\omega} \,\middle|\, \omega\in
\Omega\left( \lambda,\mu \right) \right\}$ is an $R$-basis of $\xi_\lambda
S_R\left( n,r \right)\xi_{\mu}$. Now we have the following result.

\begin{proposition}
	Under \eqref{eq:isom7} $\theta_T$, $T\in \tableau^{rs}\left( \lambda,\mu
	\right)$, corresponds to $\xi_{\omega\left( T \right)}$. 	
\end{proposition}
\begin{proof}
Let $T\in \tableau^{rs}\left( \lambda,\mu \right)$ and 
$j\in \mu$.  Then
$$\xi_{\omega\left( T \right)} e_j = \sum_{\left( i,j \right)\in\omega\left( T
\right)} e_i.$$ 
Now for all $1\le s,t\le n$
\begin{align*}
	\#\left\{\, 1\le q \le r \,\middle|\, i_q=s,\ j_q=t \right\} 
&
	= \wt\left( i,j \right)_{st} = \omega\left( T \right)_{st}.
\end{align*}
Thus $\xi_{\omega\left(T\right)}e_j$ is the sum of
$e_i$, for those $i\in \lambda$ such that $i$ is obtained from $j$ by replacing
$\omega\left( T \right)_{st}$ occurences of $t$ by $s$. 
Therefore $\mathfrak{t}\left( i
\right)$ is obtained from $\mathfrak{t}\left( j \right)$ by moving exactly
$\omega\left( T \right)_{st}$ elements from the $t$-th row of
$\mathfrak{t}\left( j \right)$ to the $s$-th row of $\mathfrak{t}\left( i
\right)$.  

Moreover, if $\mathfrak{s}\in \tableau^{rs}\left( \lambda \right)$ is obtained
from $\mathfrak{t}\left( j \right)$ by moving exactly $\omega\left( T
\right)_{st}$ elements from the $t$-th row of $\mathfrak{t}\left( j
\right)$ to $s$-th row of $\mathfrak{s}$, then for the corresponding $i = i\left(
\mathfrak{s}
\right)\in \lambda$ we have $\left( i,j \right)\in \omega\left( T
\right)_{st}$. This proves the proposition. 
\end{proof}

Boltje and Hartmann define $\Hom_{R \Sigma_r }^\wedge\left(
M^\mu,M^\lambda
\right)$ as the $R$-span of those elements $\theta_T$ with $T\in \tableau^{rs}\left(
\lambda,\mu
\right)$ such that for every $1\le s \le n$ the $s$-th row of $T$ does not
contain any entry smaller than $s$. This is equivalent to the requirement
that $\omega\left( T \right)$ is an upper triangular matrix. 
Therefore,
under \eqref{eq:isom7} the subspace $\Hom_{R \Sigma_r
}^\wedge\left( M^\mu,M^\lambda \right)$ is mapped to $\xi_{\lambda}
S^+\left( n,r \right)\xi_{\mu}$,
since we know that $\xi_\lambda S^+_R\left( n,r \right)\xi_\mu$ is the $R$-linear
span of
\begin{align*}
	\left\{\, \xi_{\omega} \,\middle|\,\mbox{$\omega$ is upper triangular
	and }  \omega\in \Omega\left( \lambda,\mu
	\right) \right\}.
\end{align*}
As $S^+_R\left( n,r \right)\cong J_R\left( n,r \right)\oplus L_{n,r}$, we have
$\Hom^\wedge_{R\Sigma_r}\left( M^\mu, M^\lambda \right)\cong \xi_\lambda
J_R\left( n,r \right)\xi_\mu$ if $\mu\strdominate \lambda$. 

In Section~3.2 of~\cite{boltje} there is defined
the complex $\widetilde{C}_*^\lambda$ 
as follows. $\widetilde{C}_{-1}^\lambda$ is the co-Specht module that
corresponds to the partition $\lambda\in \Lambda^+\left( n;r \right)$,
$\widetilde{C}_0^\lambda = \Hom_{R}\left( M^\lambda, R \right)$. For $k\ge 1$ the
$R\Sigma_r$-module $\widetilde{C}_k^\lambda$ is defined as the direct sum over all sequences
$\mu^{\left( 1 \right)}\strdominate \dots \strdominate \mu^{\left( k
\right)}\strdominate \lambda$
\begin{align}
	\label{eq:summand}
	\Hom_R\left( M^{\mu^{\left( 1 \right)}}, R \right)\otimes_R
	\Hom_{R\Sigma_r}^\wedge\left(
	M^{\mu^{\left( 2 \right)} },M^{\mu^{\left( 1 \right)}} 
	\right) \otimes_R \dots \otimes_R \Hom_{R\Sigma_r}^\wedge\left(
	M^\lambda,
	M^{\mu^{\left( k \right)}} \right). 
\end{align}
	The differential $d_k$, $k\ge 1$,  in $\widetilde{C}_*^\lambda$ is given by the formula
\begin{align}
	\label{eq:diff2}
	d_k \left( f_0 \otimes f_1 \otimes \dots \otimes f_{k-1} \right) =
	\sum_{t=0}^k
	\left( -1 \right)^t f_0 \otimes \dots \otimes f_t f_{t+1} \otimes \dots
	f_{k-1}.
\end{align}
	Note that we have arranged the factors in the definition of $\widetilde{C}_k^\lambda$ in
	a different order  from the one used  in \cite{boltje}.

\begin{theorem}
	\label{thm:boltje}
	For $\lambda\in \Lambda^+\left( n;r \right)$, 
	the complex $\widetilde{C}_*^\lambda$ is isomorphic to $\sf B_*\left( W^R_\lambda
	\right)$. 
\end{theorem}
\begin{proof}
We will establish the isomorphism only in non-negative degrees. The
isomorphism in the degree $-1$ will follow, since the complex $\sf B_*\left(
W^R_\lambda \right)$ is exact, and the complex $\widetilde{C}^\lambda_*$ is exact in
degrees $0$ and $-1$ by Theorems~4.2 and~4.3 in~\cite{boltje}.

We define the complex $\widehat{C}_*^\lambda$ in the same way as the complex
$\widetilde{C}^\lambda_*$ with the only difference that the summands \eqref{eq:summand} are
replaced by 
\begin{multline*}
	\Hom_{R\Sigma_r}\left( M^{\mu^{\left( 1 \right)}}, M^\delta
	\right)\otimes_R
	\Hom_{R\Sigma_r}^\wedge\left(
	M^{\mu^{\left( 2 \right)} },M^{\mu^{\left( 1 \right)}} 
	\right)\otimes_R\\  \otimes_R \dots \otimes_R \Hom_{R\Sigma_r}^\wedge\left(
	M^\lambda,
	M^{\mu^{\left( k \right)}} \right). 
\end{multline*}
Then it is straightforward that the isomorphism \eqref{eq:isom7} induces an 
isomorphism between the  complexes $\sf B_*\left( W^R_\lambda \right)$ and
$\widehat{C}_*^\lambda$ in non-negative degrees. 

To show that the complexes $\widehat{C}_*^\lambda$ and $\widetilde{C}_*^\lambda$ are
isomorphic in the non-negative degrees it is enough to find for every
$\nu\in \Lambda\left( n;r \right)$ an isomorphism of
$\Sigma_r$-modules $\phi_\nu\colon \Hom_{R\Sigma_r}\left(M^\nu,M^\delta
\right)\to \Hom_R\left( M^\nu,R \right)$, such that for all $\mu\in
\Lambda\left( n;r \right)$, $f\in \Hom_{R\Sigma_r}\left( M^\nu,M^\delta
\right)$, and $h\in \Hom_{R\Sigma_r}\left( M^\mu,M^\nu \right)$
we have
$\phi_\mu\left( fh  \right) = \phi_\nu\left( f \right)h$. 

Note that the action of $\Sigma_r$ on $\Hom_{R\Sigma_r}\left( M^\nu,M^\delta
\right)$ is given by composition with $\theta_{T\left( \omega\left( \sigma
\right) \right)}$, $\sigma\in \Sigma_r$. 
The action of $\Sigma_r$ on $\Hom_R\left( M^\nu, R \right)$ is given by the
formula $\left( \sigma f \right)\left( m \right) = f\left( \sigma^{-1}m
\right)$, for $f\in \Hom_R\left( M^\nu, R \right)$, $m\in M^\nu$, and
$\sigma \in \Sigma_r$. 

Let $f\in \Hom_{R\Sigma_r}\left( M^\nu, M^\delta \right)$, $m\in M^\nu$.
We define $\phi_\nu\left( f \right)\left( m \right)$ to be the coefficient
of $\mathfrak{t}_\delta$ in $f\left( m \right)$. Note that $f$ can be recovered
from $\phi_\nu \left( f \right)$ in a unique way. In fact, $\left\{\, \sigma
\mathfrak{t}_\delta \,\middle|\,
\sigma \in \Sigma_r \right\}$ is an $R$-basis of $M^\delta$. Now the coefficient of
$\sigma \mathfrak{t}_\delta$ in $f\left( m \right)$ is the same as the
coefficient of $\mathfrak{t}_\delta$ in $\sigma^{-1} f\left( m \right) = f\left(
\sigma^{-1}m \right)$, and so equals $  \phi_\nu\left( f \right)\left( \sigma^{-1}m \right)$. This
shows that $\phi_\nu$ is injective. 

Now let $g\in \Hom_R\left( M^\nu, R \right)$. We define $\psi_\nu\left(
g \right)\in \Hom_{R\Sigma_r}\left( M^\nu, M^\delta \right)$ by 
$$
m \mapsto \sum_{\sigma\in \Sigma_r} g\left( \sigma^{-1} m \right)( \sigma
\mathfrak{t}_\delta).
$$
We have to check that $\psi_\nu\left( g \right)$ is $\Sigma_r$-invariant.
Let $\sigma'\in \Sigma$. Then 
\begin{align*}
	\psi_\nu\left( g \right)\left( \sigma' m  \right)& = \sum_{\sigma\in
	\Sigma_r} g\left( \sigma^{-1} \sigma' m \right) \left( \sigma
	\mathfrak{t}_\delta \right) 
	\stackrel{\left( \sigma'' \right)^{-1} = \sigma^{-1}
	\sigma'}{\Relbar\!\Relbar\!\Relbar\!\Relbar\!\Relbar\!\Relbar\!\Relbar\!\Relbar\!\Relbar\!}
	\sum_{\sigma''\in \Sigma_r} g\left( \left( \sigma'' \right)^{-1} m
	\right)\left(\sigma' \sigma''\mathfrak{t}_\delta \right)\\
	&= 
	\sigma' \left( \psi_\nu\left( g \right) \left( m
	\right)\right).
\end{align*}
Since all the elements $\sigma \mathfrak{t}_\delta$, $\sigma\in \Sigma_r$,  are
linearly independent, we get that $\phi_\nu\left( \psi_\nu \left( g \right)
\right)\left( m \right) = g\left( m \right)$ for all $g\in \Hom_R\left(
M^\nu, R
\right)$ and $m\in M^\nu$. Therefore $\phi_\nu$ is surjective. 

Now we check that $\phi_\nu$ is a homomorphism of $\Sigma_r$-modules. For
this we have to see that for all $\sigma\in \Sigma_r$, $f\in
\Hom_{R\Sigma_r}\left( M^\nu, M^\delta \right)$, and $m\in M^\nu$ there holds
\begin{align}
	\label{eq:psi}
	\phi_\nu\left( \theta_{T\left( \omega\left( \sigma \right) \right)}f
	\right)\left( m \right) = \phi_\nu\left( f
	\right)\left(\sigma^{-1}m\right).
\end{align}
We will show that $\theta_{T\left( \omega\left( \sigma \right) \right)}$ acts by
permutation on the basis $\left\{\, \sigma'\mathfrak{t}_\delta \,\middle|\,
\sigma'\in \Sigma_r \right\}$ of $M^\delta$. Then $\phi_\nu\left(
\theta_{T\left( \omega\left( \sigma \right) \right)}f
\right)\left( m \right)$ will be the  coefficient of some
$\sigma'\mathfrak{t}_\delta$ in $f\left( m \right)$. 

Let $\sigma'\in \Sigma_r$. Then $\sigma'\mathfrak{t}_\delta$ is obtained
from $\mathfrak{t}_\delta$ by applying $\sigma'$ to every entry of
$\mathfrak{t}_\delta$. Thus $\sigma' \mathfrak{t}_\delta$ is a $\delta$-tableau 
with the entry $\sigma'\left( s \right)$ in the row $s$. 

We know that $\omega\left( \sigma \right)_{st}$ is non-zero only if $t =
\sigma^{-1}s$. Now by our reformulation of Boltje-Hartmann rule,
$\theta_{T\left( \omega\left( \sigma \right) \right)} \left( \sigma'
\mathfrak{t}_\delta \right)$ is the sum of those $\mathfrak{s}$ for which the
$s$-th row of $\mathfrak{s}$ contains precisely one entry from the
$\sigma^{-1} s$ row of $\sigma' \mathfrak{t}_\delta$. Of course such
$\mathfrak{s}$ is unique and the entry in the row $s$ is $\sigma'\left(
\sigma^{-1}s
\right)$. Thus $\theta_{T\left( \omega\left( \sigma \right) \right)}\left(
\sigma' \mathfrak{t}_\delta
\right) = \sigma'\sigma^{-1} \mathfrak{t}_\delta$.  

Therefore the coefficient of $\mathfrak{t}_\delta$ in $\theta_{T\left(
\omega\left( \sigma \right)
\right)} f\left( m \right)$ is the same as coefficient of $\sigma
\mathfrak{t}_\delta$ in $f\left( m \right)$, which is also the
coefficient of $\mathfrak{t}_\delta$ in $f\left( \sigma^{-1}m \right)$. 
So \eqref{eq:psi} holds. 

It is left to check that for every $\nu$, $\mu\in \Lambda\left( n;r
\right)$, and $h\in \Hom_{R\Sigma_r}\left( M^\mu,M^\nu \right)$, $f\in
\Hom_{R\Sigma_r}\left( M^\nu, M^\delta \right)$, we have $\phi_\mu\left(
fh \right) = \phi_\nu \left( f \right) h$.  But this is
immediate since both
$\phi_\mu\left( fh \right)\left( m \right)$ and $\phi_\nu\left( f
\right)\left( h\left( m \right) \right)$ are the coefficients of
$\mathfrak{t}_\delta$ in $fh\left( m \right)$, for every $m\in M^{\mu}$.  
\end{proof}

\appendix

\section{Appendix}
\label{genlin}
In this appendix we discuss  the connection between the general linear group and Schur
algebras as they are defined in this paper. We start with the proof that our
definition of Schur algebra is equivalent to the one given in Green~\cite{greencombinatorics}.

We define $A_R\left( n \right)$ to be the commutative ring $R\left[ c_{s,t}\colon
s,t\in \bfn
\right]$ in the indeterminates $c_{st}$.

For every $i$, $j\in I\left( n,r \right)$ we denote by $c_{i,j}$ the product
$c_{i_1,j_1}\dots c_{i_r,j_r}$. Then $\left( i,j \right)$ and $\left( i',j'
\right)\in I\left( n,r \right)\times I\left( n,r \right)$ are on the same
$\Sigma_r$-orbit if and only if $c_{i,j} = c_{i',j'}$. We will write
$c_{\omega}$ for  $c_{i,j}$ if $\left( i,j \right)\in \omega$.

Denote by $A_R\left( n,r \right)$ the $R$-submodule of $A_R\left( n
\right)$
of all homogeneous polynomials of degree $r$ in the $c_{st}$. Then $\left\{\,
c_\omega \,\middle|\, \omega\in \Lambda\left( n,n;r \right) \right\}$ is an
$R$-basis of $A_R\left( n,r \right)$ and $A_n = \bigoplus_{r\ge 0} A_R\left(
n,r \right)$. It is well known (see~\cite{greencombinatorics}) that $A_R\left( n,r
\right)$ has the  structure of 
a  coassociative
coalgebra with the structure maps given by 
\begin{align*}
	\triangle \left( c_{i,j} \right) &:= \sum_{k\in I\left( n,r \right)}
	c_{i,k}\otimes c_{k,j},  & \varepsilon\left( c_{i,j} \right)&:=
	\delta_{i,j}.
\end{align*}
	\emph{The Schur algebra } $S^{Gr}_R\left( r,n \right)$ in the sense of
	Green is the $R$-algebra dual to the coalgebra $A_R\left( n,r
	\right)$. Let $\left\{\, \widehat{\xi}_\omega \,\middle|\, \omega\in
	\Lambda\left( n,n;r \right) \right\}$ be the $R$-basis of
	$S^{Gr}_R\left( n,r \right)$ dual to $\left\{\, c_\omega \,\middle|\,
	\omega\in \Lambda\left( n,n;r \right)
	\right\}$. 
	Define a  map $\phi$ from $S^{Gr}_R\left( n,r \right)$ to $\End_R\left(
	\left( R^n \right)^{\otimes r} \right)$ by 
	$$
	\phi\left( f \right) e_j := \sum_{i\in I\left( n,r \right)} f\left(
	c_{ i,j }
	\right) e_i.
	$$
	\begin{theorem}
		\label{equivalence}
		The map $\phi$ provides  an $R$-algebra isomorphism between
		$S^{Gr}_R\left( n,r \right)$ and $S_R\left( n,r \right)$.
	\end{theorem}
	\begin{proof}
		Clearly $\phi$ maps the identity of $S^{Gr}_R\left( n,r
		\right)$ into the identity map of $\End_R\left( \left( R^n
		\right)^{\otimes r} \right)$. 
		Let $f_1$, $f_2\in
		S^{Gr}_R\left( n,r \right)$. Then the product
		(see~\cite{green} or~\cite{greencombinatorics}) of $f_1$ and
		$f_2$ is defined by 
		$$
		\left( f_1f_2 \right)\left( c_{i,j} \right) = \sum_{k\in I\left(
		n,r \right)} f_1\left( c_{i,k} \right)f_2 \left( c_{k,j}
		\right). 
		$$
		Therefore
		\begin{align*}
			\phi\left( f_1f_2 \right) e_j& = \sum_{i\in I\left( n,r
			\right)} \sum_{k\in I\left( n,r \right)} f_1\left(
			c_{i,k}
			\right) f_2\left( c_{k,j} \right) e_i. 
		\end{align*}
		On the other hand
		\begin{align*}
			\phi\left( f_1 \right)\phi\left( f_2 \right) e_j &= 
			\phi\left( f_1 \right) \sum_{k\in I\left( n,r
			\right)} f_2\left(
			c_{k,j}
			\right)e_k = \sum_{k\in I\left( n,r \right)} \sum_{i\in
			I\left( n,r \right)} f_2\left(
			c_{k,j}
			\right)f_1\left( c_{i,k} \right) e_i.
		\end{align*}
		Therefore $\phi\left( f_1f_2 \right) = \phi\left( f_1
		\right)\phi\left( f_2 \right)$. 
	We have also $\phi\left( \widehat{\xi}_{\omega} \right) = \xi_{\omega}$.
	In fact
	\begin{align*}
		\phi\left( \widehat{\xi}_{\omega} \right) e_j & =
		\sum_{i\in I\left( n,r \right)} \widehat{\xi}_{\omega}\left(
		c_{i,j} \right) e_i = \sum_{i\in I\left( n,r \right)\colon
		\left( i,j \right)\in \omega} e_i = \xi_{\omega} e_j. 
	\end{align*}
	Thus the result follows. 
	\end{proof}
	To each $c_{s,t}\in A_R\left( n \right)$ we can associate a function
	$\widetilde{c}_{s,t}\colon \GL_n\left( R \right)\to R$ defined by
	$\widetilde{c}_{s,t}(g) =
	g_{st}$, $g = \left( g_{s,t} \right)_{s,t=1}^n\in \GL_n\left( R \right)$. The correspondence
	$c_{s,t}\mapsto \widetilde{c}_{s,t}$ induces an algebra homomorphism
	$\psi$
	from $A_R\left( n \right)$ to the algebra of maps from $\GL_n\left( R
	\right)$ to $R$.  The homomorphism $\psi$ is injective if $R$ is an
	infinite field, but this is not true for arbitrary commutative rings or
	even finite fields. We write
	$\widetilde{c}_{\omega}$  for $\psi\left( c_{\omega} \right)$,
	$\omega\in \Lambda\left( n,n;r \right)$. Then for all $g\in \GL_n\left(
	R \right)$ and $\left( i,j \right)\in \omega$ we have  
\begin{align}
	\label{eq:comega}
	\widetilde{c}_{\omega}\left( g \right)& = \prod_{s,t=1}^n
	g_{st}^{\omega_{st}},&
	\widetilde{c}_{i,j}\left( g \right)  =
	\prod_{q=1}^r g_{i_q,j_q} . 
\end{align}

Next we give an overview of the functor of divided powers and explain
the 
relation between divided powers and principal projective modules over the Schur
algebra $S_R\left( n,r \right)$. 

 We start by recalling the definition of the algebra of divided powers
$D\left( M \right)$ associated to an $R$-module $M$. 
A good reference on this subject is~\cite{roby}. 

We denote by
$\mathcal{D}\left( M \right)$ the free algebra
over $R$ generated by variables $X_{\left(m,k\right)}$, $m\in M$ and $k\in\N$. Now $D\left(
M \right)$ is defined as the quotient of $\mathcal{D}\left( M \right)$ by the ideal generated by
\begin{align}
	\nonumber X_{\left(m,0\right)} &- 1 && \left( m\in M \right)\\
\nonumber 	X_{\left( a m, k \right)} &- a^k X_{\left( m,k \right)} &&
	\left( a\in R,\ m\in M,\ k\ge 0  \right)\\
	\label{eq:ideal}	X_{\left( m,k \right)} X_{\left( m,l \right)} &- \binom{k+l}{k}
	X_{\left( m,k+l \right)} && \left( m\in M,\ k,\ l\ge 0 \right)\\
\nonumber		X_{\left( m_1+m_2,l \right)} &- \sum_{k=0}^l X_{\left( m_1,k
	\right)}X_{\left(m_2,l-k\right)}  && \left( m_1,\ m_2\in M,\ l\ge 0 \right). 
\end{align}

We will denote the image of $X_{\left( m,k \right)}$ in $D\left( M
\right)$ by $m^{\left( k \right)}$. 
Then, besides the relations determined by \eqref{eq:ideal}, it also holds
\begin{align*}
	0^{\left( k \right)} &=0, && k\ge 1\\
	m^{\left( k_1 \right)}\dots m^{\left( k_s \right)}
	&=\binom{k_1+\dots+k_s}{k_1,\dots,k_s} m^{ \left( k_1+\dots+k_s
	\right) } ,&& m\in M;\ k_1,\dots,k_s\ge 0\\[2ex]
	\left( m_1+\dots + m_s \right)^{\left( k \right)}& = \sum_{k_1+\dots +
	k_s = k} m_1^{\left( k_1 \right)}\dots m_s^{\left( k_s \right)} ,&& m_1,\dots,m_s\in M; k\ge 0. 
\end{align*}
If $M$ and $N$ are two $R$-modules and $f\colon M\to N$ is a homomorphism of
$R$-modules, we can define an $R$-algebra homomorphim $D\left( f
\right)\colon D\left( M \right)\to D\left( N \right)$  by $D\left( f
\right)\left( m^{\left( k \right)} \right) = \left(f\left( m \right)\right)^{\left( k
\right)}$, $m\in M$, $k\ge 0$.  Note that the rings $\mathcal{D}\left(
M \right)$ are graded, with $\mathop{deg}\left( X_{\left( m,k \right)} \right)=
k$, $m\in M$, $k\ge 0$. As all the relations \eqref{eq:ideal} are homogeneous,
the ring $D\left( M \right)$ is graded as well. Now the map $D\left( f \right)$
preserves the grading for any map of
$R$-modules $f$. Therefore $D$ is a functor from the category of $R$-modules to
the category of graded $R$-algebras. We denote by $D_k\left( M \right)$ the
$k$-th homogeneous component of $D\left( M \right)$. Given a homomorphism of
$R$-modules $f\colon M\to N$ we define $D_k\left( f
\right)$  to be the restriction of $D\left( f \right)$ to $D_k\left( M
\right)$. In this way we obtain the endofunctor $D_k$ on the category of
$R$-modules.  

Define the map $\tau_k\colon \GL\left( M \right)\to \End_R\left( D_k\left( M
\right)\right)$ by $\tau_k\left( f \right) := D_k\left( f \right)$. This map is
obviously multiplicative and thus extends to a  representation $\tau_k\colon
R\GL\left( M \right) \to \End_R\left( D_k\left( M \right)
\right)$ of
the group $\GL\left( M \right)$.
Thus we get a structure of $\GL\left( M \right)$-module on $D_k\left( M
\right)$. As the category of $\GL\left( M \right)$-modules is monoidal we can
define
for every $\lambda = \left( \lambda_1,\dots, \lambda_n \right)\in \Lambda\left(
n;r \right)$ the $\GL\left(
M \right)$-module $D_\lambda\left( M \right)$ by
$$
D_\lambda\left( M \right) = D_{\lambda_1}\left( M \right)\otimes_R \dots
\otimes_R D_{\lambda_n}\left( M \right).
$$
Now we shall give a more explicit description of $D_\lambda\left( M \right)$ in
the case $M = R^n$.

Recall that $\left\{ e_1,\dots, e_n
\right\}$ denotes the standard basis of $R^n$. By Theorem~IV.2 in~\cite{roby} 
$$
\left\{\, e_1^{\left( k_1 \right)}\dots e_n^{\left( k_n \right)} \,\middle|\,
k_t\in \N, \sum_{t=1}^n k_t = k \right\}
$$
is a basis of $D_k\left( R^n\right)$ for $k\ge 0$. 
Thus
the set
	$$
	\left\{\, e^{\left( \pi \right)} \,\middle|\, \pi\in \Lambda\left( n,n;r
	\right), \pi^1 = \lambda\right\}, 
	$$
	where $e^{\left( \pi \right)} = \bigotimes_{t=1}^n \prod_{s=1}^n
	e_s^{\left( \pi_{st} \right)}$, 
	is an $R$-basis of the $R$-module $D_\lambda\left( R^n \right)$. 
\begin{proposition}
	\label{actiondivided}
	Let $\lambda\in \Lambda\left( n;r \right)$. Then the action of
	$\GL_n\left( R \right)$ on $D_\lambda\left( R^n \right)$ is given by 
	\begin{align*}
		ge^{\left( \pi \right)} &= \sum_{\theta\in \Lambda\left( n,n,n;r
		\right)\colon \theta^1 = \pi} \left[ \theta
		\right]\widetilde{c}_{\theta^3}\left( g \right)e^{\left( \theta^2 \right)},& \pi\in
		\Lambda\left( n,n;r \right),\ \pi^1 = \lambda,\ g\in \GL_n\left(
		R \right). 
	\end{align*}
\end{proposition}
\begin{proof}
	Let $g=\left( g_{st} \right)_{s,t=1}^n \in \GL_n\left( R \right)$ and
	$\pi\in \Lambda\left( n,n;r \right)$ such that $\pi^1 =\lambda$. Then
	\begin{align*}
		g e^{\left( \pi \right)} &= \bigotimes_{t=1}^n \prod_{s=1}^n
		\left( ge_s \right)^{\left(\pi_{st}\right)} = \bigotimes_{t=1}^n
		\prod_{s=1}^n \left( \sum_{q=1}^n g_{qs} e_q \right)^{\left(
		\pi_{st} \right)}\\
		&= \bigotimes_{t=1}^n \prod_{s=1}^n
		\sum_{\nu\in \Lambda\left( n;\pi_{st} \right)}
		\prod_{q=1}^{ n} g_{q,s}^{\nu_q} e_q^{\left( \nu_q \right)}
		\\& =
		\sum_{\nu\colon \bfn\times \bfn \to \Lambda\left(
		n
		\right)}  \bigotimes_{t=1}^n \prod_{s,q=1}^n
		g^{\nu\left( s,t \right)_q}_{q,s} e_q^{\left(\nu\left( s,t
		\right)_q\right)},
	\end{align*}
	where the summation is over the set $X\left( \pi \right)$ of functions
	\begin{align*}	
		\nu\colon \bfn\times \bfn
		\to \Lambda\left( n \right) = \bigcup_{r\ge 0 }\Lambda\left( n;r
		\right)
		\end{align*}
		such that $\nu\left( s,t \right)\in
	\Lambda\left( n;\pi_{st} \right)$. 
	There is a one-to-one correspondence between $X\left( \pi \right)$ and
	the set $Y\left(\pi  \right) :=\left\{\, \theta\in \Lambda\left( n,n,n;r
	\right) \,\middle|\, \theta^1 = \pi
	\right\}$. 
	In fact, let $\theta\in Y\left( \pi \right)$. Then we can define
	$\nu\colon \bfn\times\bfn \to \Lambda\left( n \right)$ by
	$$
	\left( s,t \right)\mapsto \left( \theta_{1st}, \dots, \theta_{nst}
	\right).
	$$
	Since $\left( \theta^1 \right)_{st} = \pi_{st}$ we get $\nu\left(
	s,t \right)\in \Lambda\left( n;\pi_{st} \right)$ for all $1\le s,t\le
	n$. 
	Now let $\nu\in X\left( \pi \right)$. Then we can define $\theta\in
	\Lambda\left( n,n,n;r \right)$ by $\theta_{qst} = \nu\left( s,t
	\right)_q$. Since $\nu\left( s,t \right)\in \Lambda\left(
	n;\pi_{st}
	\right)$ we get that $\theta^1 = \pi$. Thus $\theta\in Y\left( \pi
	\right)$. It is easy to see that these constructions are mutually
	inverse. 
	Therefore
	\begin{align*}
		ge^{\left( \pi \right)} & = \sum_{\theta\in \Lambda\left(
		n,n,n;r
		\right)\colon \theta^1 = \pi}\ \ \ \bigotimes_{t=1}^n \prod_{s,q=1}^n g_{q,s}^{\theta_{qst}} e_{q}^{\left(
		\theta_{qst}
		\right)}.
	\end{align*}
	Since $R$ is a commutative ring it is left to show that for every
	$\theta\in Y\left( \pi \right)$ we have
	\begin{align*}
		\prod_{t,s,q=1}^n
		g_{qs}^{\theta_{qst}} &= \widetilde{c}_{\theta^3}\left( g
		\right), &
		\bigotimes_{t=1}^n \prod_{s,q=1}^n
		e_{q}^{\left(\theta_{qst}\right)} &= \left[ \theta \right]e^{\left( \theta^2
		\right)}.
	\end{align*}
	For the first formula we have
	\begin{align*}
			\prod_{t,s,q=1}^n
		g_{qs}^{\theta_{qst}} = 
  \prod_{s,q=1}^n g_{qs}^{\sum_{t=1}^n \theta_{qst}}
=  \prod_{s,q=1}^n g_{qs}^{\left( \theta^3 \right)_{qs}}
\stackrel{\eqref{eq:comega}}{\Relbar} \widetilde{c}_{\theta^3}\left( g \right). 
	\end{align*}
For the second identity we fix $t\in \bfn$ and compute the product
\begin{align*}
	\prod_{s,q=1}^n e_q^{\left( \theta_{qst} \right)} = 
	\prod_{q=1}^n \left(\prod_{s=1}^n e_q^{\left( \theta_{qst}
	\right)}\right) = 
	\prod_{q=1}^n \binom{\left( \theta^2 \right)_{qt}}{\theta_{q1t}, \dots,
	\theta_{qnt}} e_q^{\left(\left( \theta^2 \right)_{qt}\right)}.
\end{align*}
Now the result follows.
\end{proof}
The group $\GL_n\left( R \right)$ acts naturally on $R^n$, by multiplication.
Thus $\GL_n\left( R \right)$ acts on $\left( R^n \right)^{\otimes r}$
diagonally. Denote by $\rho_{n,r}$ the corresponding representation of
$\GL_n\left( R \right)$. 
 Since for every $g\in \GL_n\left( R
\right)$ the endomorphism $\rho_{n,r}\left( g \right)$ commutes with the action
of $\Sigma_r$ on $\left( R^n \right)^{\otimes r}$, we get that $\im\left(
\rho_{n,r}
\right)\subset S_R\left(n,r \right)$.

\begin{proposition}
	\label{imagerho}
	Let $g\in\GL_n\left( R \right)$. Then $\rho_{n,r}\left( g \right) =
	\sum_{\omega\in\Lambda\left( n,n;r \right)} \widetilde{c}_{\omega}\left( g
	\right) \xi_{\omega}$. 
\end{proposition}
\begin{proof}
	Let $i\in I\left( n,r \right)$. Then we have
	\begin{align*}
		ge_i&  = ge_{i_1} \otimes \dots \otimes ge_{i_r}
		\\& =
		\sum_{j\in I\left( n,r \right)} \left( g_{j_1,i_1}\dots
		g_{j_r,i_r} \right) e_{j_1}\otimes \dots \otimes e_{j_r} 
		\\& \stackrel{\eqref{eq:comega}}{\Relbar}\sum_{j\in I\left( n,r
		\right)} \widetilde{c}_{\wt\left( j,i
		\right)} \left( g \right) e_j = \sum_{\omega\in \Lambda\left(
		n,n;r \right)} \widetilde{c}_\omega\left( g \right) \sum_{j\colon \left( j,i
		\right)\in \omega} e_j = \sum_{\omega\in \Lambda\left( n,n;r
		\right)} \widetilde{c}_\omega\left( g \right) \xi_\omega e_i.
	\end{align*}
	As the set $\left\{\, e_i \,\middle|\, i\in I\left( n,r \right)
	\right\}$ is a basis of $\left( R^n \right)^{\otimes r}$, we see that
	$\rho_{n,r}\left( g \right) = \sum_{\omega\in \Lambda\left( n,n;r
	\right)} \widetilde{c}_\omega\left( g \right)\xi_{\omega}$. 
\end{proof}

\begin{theorem}
	\label{isodivided}
	Let $\lambda\in \Lambda\left( n,r \right)$ and
	 consider $S_R\left( n,r \right)\xi_{\lambda}$ as a $\GL_n\left( R
	 \right)$-module via the homomorphism $\rho_{n,r}$. 
Then the map 
\begin{align*}
	\psi\colon D_{\lambda}\left( R^n \right) & \to S_R\left( n,r
	\right)\xi_{\lambda}\\
	e^{\left( \pi \right)}& \mapsto \xi_{\pi}
\end{align*}
is an isomorphism of $\GL_n\left( R \right)$-modules. 
\end{theorem}
\begin{proof}
	It is clear that $\psi$ is an isomorphism of $R$-modules. Thus it is
	enough to show that $\psi$ commutes with the action of $\GL_n\left( R
	\right)$. 
	 Let $g\in \GL_n\left( R \right)$ and $\pi\in \Lambda\left(
	n,n;r \right)$, $\pi^1 = \lambda$. Then 
	by Proposition~\ref{actiondivided} we have
	\begin{align*}
		\psi\left( ge^{\left(\pi\right)} \right) = \sum_{\theta\in \Lambda\left(
		n,n,n;r
		\right)\colon \theta^1 = \pi} \left[ \theta
		\right]\widetilde{c}_{\theta^3}\left( g \right) \psi\left( e^{\left(
		\theta^2
		\right)} \right) =  \sum_{\theta\in \Lambda\left(
		n,n,n;r
		\right)\colon \theta^1 = \pi} \left[ \theta
		\right]\widetilde{c}_{\theta^3}\left( g \right)\xi_{\theta^2} . 
	\end{align*}
	On the other hand by Proposition~\ref{imagerho} and
	Proposition~\ref{mult}
	\begin{align*}
		g\psi\left( e^{\left( \pi \right)} \right) & = \left(
		\sum_{\omega\in \Lambda\left( n,n;r \right)} \widetilde{c}_{\omega}\left( g
		\right)\xi_{\omega}
		\right) \xi_{\pi} \\[2ex]& = \sum_{\omega\in \Lambda\left( n,n;r
		\right)} \widetilde{c}_\omega\left( g \right) \sum_{\theta\in \Lambda\left(
		n,n,n;r \right)\colon \theta^3= \omega, \ \theta^1 = \pi}
		\left[ \theta \right]\xi_{\theta^2}
		\\[2ex]&
		= \sum_{\theta\in \Lambda\left( n,n,n;r
		\right)\colon \theta^1 = \pi}\left[ \theta \right] \widetilde{c}_{\theta^3}\left( g
		\right) \xi_{\theta^2}. 
	\end{align*}
\end{proof}

\bibliography{schur}
\bibliographystyle{amsalpha}

\end{document}